\documentclass[a4paper,reqno]{amsart}
\usepackage[T1]{fontenc}
\usepackage[utf8x]{inputenc}
\usepackage[english]{babel}
\usepackage{amsmath,amsfonts,amssymb,upgreek,comment}
\usepackage{mathrsfs}
\usepackage{graphicx}
\usepackage{hyperref}
\usepackage{enumerate}
\usepackage{xcolor}
\usepackage{esint}
\usepackage{amsthm}

\theoremstyle{plain}
\newtheorem*{theorem*}{Theorem}
\newtheorem{theorem}{Theorem}[section]
\theoremstyle{definition}
\newtheorem{D}[theorem]{Definition}
\newtheorem{lemma}[theorem]{Lemma}
\newtheorem{cor}[theorem]{Corollary}
\newtheorem{prop}[theorem]{Proposition}

\theoremstyle{definition}
\newtheorem{rem}[theorem]{Remark}

\newcommand{\R}{\ensuremath{\mathbb R}}
\newcommand{\N}{\ensuremath{\mathbb N}}

\newcommand{\oG}{\overline{G}}
\newcommand{\eps}{\ensuremath{\varepsilon}}

\newcommand{\Ric}{\ensuremath{\mbox{Ric}}}
\newcommand{\Ricm}{\ensuremath{\mbox{Ric}_{\mbox{\tiny{-}}}}}

\newcommand{\measrestr}{%
  \,\raisebox{-.127ex}{\reflectbox{\rotatebox[origin=br]{-90}{$\lnot$}}}\,%
}

\newcommand{\setN}{\mathbb{N}}

\newcommand{\setR}{\mathbb{R}}

\newcommand{\cD}{\mathcal{D}}

\newcommand{\mE}{\mathrm{E}}

\newcommand{\setB}{\mathbb{B}}

\newcommand{\weakto}{\rightharpoonup}

\newcommand{\Id}{\mathrm{Id}}

\newcommand{\di}{\mathop{}\!\mathrm{d}}

\DeclareMathOperator{\supp}{supp}

\newcommand{\Ch}{{\sf Ch}}
\newcommand{\kato}{\mathrm{k}}
\newcommand{\katot}{\ensuremath{\mbox{k}_t(M^n,g)}}

\DeclareMathOperator{\Lip}{Lip}

\newcommand{\leb}{\mathcal{H}}

\newcommand{\dist}{\mathsf{d}}

\newcommand{\meas}{\mathfrak{m}}

\DeclareMathOperator{\RCD}{RCD}

\newcommand{\vol}{\nu}
\DeclareMathOperator{\Hess}{Hess}

\newfont{\tmpf}{cmsy10 scaled 2500}

\def\R{\mathbb R}\def\N{\mathbb N}\def\bB{\mathbb B}
\def\cB{\mathcal B}
\def\cC{\mathcal C}
\def\cD{\mathcal D}

\def\cH{\mathcal H}

\def\cM{\mathcal M}

\def\cR{\mathcal R}
\def\cS{\mathcal S}\def\cK{\mathcal K}

\def\cU{\mathcal U}

\def\un{\mathbf{1}}

\newcommand{\katolimits}{\overline{\cK_\meas(n,f,c)}}
\newcommand{\strongkatolimits}{\overline{\cK(n,f,v)}}

\DeclareMathOperator{\Tan}{Tan}
\DeclareMathOperator{\dGH}{\dist_{GH}}
\DeclareMathOperator{\pGH}{pGH}

\DeclareMathOperator{\dmGH}{\dist_{mGH}}

\usepackage{xcolor}

\title[Rectifiability and $C^{1,1}$-structure]{Limits of manifolds with a Kato bound on the Ricci curvature. II.}

\author{Gilles Carron}
\address{G. Carron, Nantes Université, CNRS, Laboratoire de Mathématiques Jean Leray, LMJL, UMR 6629, F-44000 Nantes, France.} 
\email{Gilles.Carron@univ-nantes.fr}

\author{Ilaria Mondello}
\address{I. Mondello, Université Paris Est Cr\'eteil, Laboratoire d'Analyse et Math\'ematiques appliqu\'es, UMR CNRS 8050, F-94010 Creteil, France.}
\email{ilaria.mondello@u-pec.fr}

\author{David Tewodrose}
\address{D. Tewodrose, Nantes Université, CNRS, Laboratoire de Mathématiques Jean Leray, LMJL, UMR 6629, F-44000 Nantes, France.}
\email{David.Tewodrose@univ-nantes.fr}
\date{}

\begin{document}
%\nocite{*} 

\maketitle

\begin{abstract}
We prove that metric measure spaces obtained as limits of closed Riemannian manifolds with Ricci curvature satisfying a uniform Kato bound are rectifiable.  In the case of a non-collapsing assumption and a strong Kato bound,  we additionally show that for any $\upalpha \in (0,1)$ the regular part of the space lies in an open set with the structure of a $\cC^\upalpha$-manifold.
\end{abstract}

\tableofcontents

\section{Introduction}

In this paper, we establish new geometric and analytic properties of Kato limit spaces, i.e.~measured Gromov-Hausdorff limits of closed Riemannian manifolds with Ricci curvature satisfying a uniform Kato bound.  Our work continues the study began in \cite{CMT} where we introduced these spaces.

For a closed Riemannian manifold $(M^n,g)$ of dimension $n \ge 2$, define
\begin{equation*}
 \katot := \sup_{x \in M}\int_0^t\int_M H(s,x,y)\Ricm(y) \di \nu_g(y) \di s
\end{equation*}
for any $t>0$, where $H$ is the heat kernel of $M$,  $\nu_g$ is the Riemannian volume measure and $\Ricm : M \to \setR_+$ is the lowest non-negative function such that for any $x \in M$,
\[
\Ric_x  \ge - \Ricm(x)g_x.
\]
Equivalently, $\Ricm$ is the negative part of the smallest eigenvalue of the Ricci tensor.

For the whole article,  we keep a positive number $T$ and a function \mbox{$f: (0,T] \to \R_+$} fixed, so that $f$ is non-decreasing and
\begin{equation}\label{eq:f}\lim_{t \to 0} f(t)=0 \quad \text{and} \quad f(T) \leq \frac{1}{16n} \cdot\end{equation}
We let $\mathcal{K}(n,f)$ be the set of isometry classes of $n$-dimensional closed Riemannian manifolds $(M^n,g)$ satisfying the Kato bound
\begin{equation}\label{eq:Kato}\tag{K}
\kato_t(M^n,g)\leq f(t), \qquad \forall t \in (0,T].
\end{equation}
This bound is implied, for instance, by a lower bound on the Ricci curvature, or by a suitable uniform $L^p$ estimate on $\Ricm$ \cite{RoseStollmann}.

For $c>0$ fixed throughout the article, let $\mathcal{K}_\meas(n,f,c)$ be the set of quadruples $(M^n,\dist_g,\mu,o)$ where $(M^n,g) \in \mathcal{K}(n,f)$, $o \in M$, $\dist_g$ is the Riemannian distance associated with $g$ and $\mu$ is a multiple of $\nu_g$ satisfying
\[
c \le \mu(B_{\sqrt{T}}(o)) \le c^{-1}.
\]
As proved in \cite{C16,CMT}, elements in $\mathcal{K}(n,f)$ satisfy a uniform doubling condition. As a consequence,  Gromov's precompactness theorem ensures that the set $\mathcal{K}_\meas(n,f,c)$ is precompact in the pointed measured Gromov-Hausdorff topology.  We call Kato limit space any element in the closure $\katolimits$ with respect to this topology. Observe that Ricci limit spaces, that is limits of manifolds with a uniform Ricci lower bound \cite{ChCo97,CheegerColdingII,CheegerColdingIII,CheegerPisa}, are Kato limit spaces.

Our first result is the rectifiability of Kato limit spaces. This was shown for Ricci limit spaces in \cite[Theorem 5.7]{CheegerColdingIII}. 

\begin{theorem}\label{th:main1}
Let $(X,\dist,\mu,o)$ be a Kato limit space.  Then $(X,\dist,\mu)$ is rectifiable as a metric measure space, in the sense that there exists a countable collection $\{(k_i,V_i,\phi_i)\}_i$ where $\{V_i\}$ are Borel subsets covering $X$ up to a $\mu$-negligible set, $\{k_i\}$ are positive integers, and $\phi_i : V_i \to \setR^{k_i}$ is a bi-Lipschitz map such that $(\phi_i)_\#(\mu \measrestr V_i) \ll \cH^{k_i}$ for any $i$, where $\cH^{k_i}$ is the $k_i$-dimensional Hausdorff measure.
\end{theorem}

Consider now the non-collapsing case, that is there exists $v>0$ such that for some $o \in M$
\begin{equation}\label{eq:NC}\tag{NC}
\vol_g(B_{\sqrt{T}}(o)) \geq v T^{\frac{n}{2}}.
\end{equation}
Assume that $f$ additionally satisfies
\begin{equation}\label{eq:SK}\tag{SK}
\int_{0}^T \frac{\sqrt{f(t)}}{t}\di t < +\infty. 
\end{equation}
In this case,  we say that $(M^n,g) \in \mathcal{K}(n,f)$ satisfies a strong Kato bound. Let $\mathcal{K}(n,f,v)$ be the set of isometry classes of pointed  closed $n$-dimensional manifolds $(M^n,g,o)$ satisfying a strong Kato bound and the non-collapsing assumption. We call non-collapsed strong Kato limit space any element in the closure $\strongkatolimits$ with respect to the pointed Gromov-Hausdorff topology. Notice that we do not need to consider measured Gromov-Hausdorff topology, because, thanks to the volume continuity proved in \cite[Theorem 7.1]{CMT}, Riemannian volumes converge to the $n$-dimensional Hausdorff measure.

Our second main result is the bi-Hölder regularity of the regular set of non-collapsed strong Kato limit spaces. This was proved for non-collapsed Ricci limit spaces in \cite[Theorem 5.14]{ChCo97}. 

\begin{theorem}\label{th:Holderreg}
Let $(X,\dist,o)$ be a non-collapsed strong Kato limit space.  Then for any $\upalpha\in (0,1)$ the regular set
\[
\cR:=\{ x \in X : (\R^n,\dist_e,0) \in \Tan(X,x)\}
\]
is contained in an open $\cC^\upalpha$ manifold $\cU_\alpha \subset X$. Here $\dist_e$ is the Euclidean distance and $\Tan(X,x)$ is the set of metric tangent cones of $X$ at $x$, see Definition \ref{D:tangent}.
\end{theorem}

In \cite[Theorem 6.2]{CMT} we also showed that non-collapsed strong Kato limit spaces admit a stratification. By combining this with volume continuity and arguments from \cite[Theorem 6.1]{ChCo97} (see also \cite[Theorem 10.22]{CheegerPisa}), we then prove that the singular set $\cS := X \setminus \cR$ of any $(X,\dist, o) \in \overline{\cK(n,f,v)}$ has codimension two. For the sake of completeness, we provide a proof in the Appendix. 

Our proofs of Theorem \ref{th:main1} and Theorem \ref{th:Holderreg} strongly rely on the existence of splitting maps on Kato limit spaces. These are harmonic maps with a suitable \mbox{$W^{2,2}$-estimate} which realize a Gromov-Hausdorff approximation between a small ball around $x$ and a Euclidean ball of same radius.  In Section 3, we give conditions for the existence of such maps, and establish some of their properties,  relying on the analysis performed in \cite{CMT}.

In order to prove Theorem \ref{th:main1}, we start by observing that almost splitting maps exist around any point $x$ of a Kato limit space admitting a Euclidean tangent cone.  After that, by means of a suitable propagation property of these maps, we adapt arguments from \cite{BPS} which built upon \cite{GigliPas} to provide a proof of the rectifiability of $\RCD(K,N)$ spaces \cite{MondinoNaber} via almost splitting maps.  Let us point out that, unlike the uniform lower Ricci bound considered in \cite{ChCo97}, the Kato bound \eqref{eq:Kato} does not provide a directionally restricted relative volume comparison on the limit space,  so that the proof of rectifiability by Cheeger and Colding, based on a suitable control on the volume deformation of pseudo-cubes through pseudo-translations, do not carry out.

We do not know whether the dimensions $k_i$ in Theorem \ref{th:main1} are all equal to a constant. This has been proved true by Colding and Naber for Ricci limit spaces \cite{ColdingNaber}, but the arguments used there seem hardly reproducible in the context of a Kato bound; a more conceivable approach may be the one used by Brué and Semola in the $\RCD(K,N)$ framework \cite{BrueSemola}.

To prove Theorem \ref{th:Holderreg}, a key tool is the following almost monotone quantity, which we introduced in \cite{CMT} to get information on the infinitesimal geometry of non-collapsed strong Kato limits. For $X \in \strongkatolimits$, $x \in X$, $t>0$,  consider
\[
\uptheta(t,x) := (4\pi t)^{n/2} H(t,x,x)
\]
where $H$ is the heat kernel of $X$.  In case $(M^n,g)$ is a Riemannian manifold with non-negative Ricci curvature, the Li-Yau inequality implies that the function \mbox{$t \mapsto \uptheta(t,x) \in [1, +\infty[$} is non-decreasing for all $x \in M$.  When $(M^n,g)$ satisfies a strong Kato bound, we showed in \cite{CMT} that this function is almost non-decreasing everywhere.  In particular, its limit as $t$ goes to zero is well-defined, not less than one, and coincides with the inverse of the volume density at $x$.  In the present paper, we prove that under \eqref{eq:SK} the regular set of $X$ is given by points where the limit of $\uptheta$ equals one:
\[
\cR=\{ x \in X : \lim_{t\to 0}\uptheta(t,x) = 1\}.
\]

We also establish that if $\uptheta(t,x)$ is close enough to 1 for some $t>0$ and $x \in X$, then any ball centered around $x$ with small radius is Gromov-Hausdorff close to a Euclidean ball with same radius.  More precisely, we prove the following Reifenberg regularity statement.

\begin{theorem}\label{thm:ReifregX}
Assume that \eqref{eq:SK} holds.  Then for any $\eps>0$ there exists $\delta>0$ depending on $n$, $f$ and $\eps$ such that for any $(X,\dist,o)\in \strongkatolimits$,  if $x \in X$ and $t \in (0,\delta T)$ satisfy
\begin{equation}
\label{eq:assum}
\uptheta(t,x) \le 1+\delta
\end{equation}
then for any $y\in B_{\sqrt{t}}(x)$ and $s\in (0,\sqrt{t}]$,
$$\dGH\left( B_{s}(y),\bB^n_{s}\right)\le \eps s,$$ 
where $\setB_s^n$ is the Euclidean ball of radius $s$ centered at $0\in \R^n$.
\end{theorem}

In addition to the almost monotonicity of $\uptheta$ and the appropriate Li-Yau inequality for Kato limit spaces (see Proposition \ref{lem:LY}), a salient ingredient in our proof of Theorem \ref{thm:ReifregX} is the heat kernel rigidity result obtained in \cite{CT}, which allows for a suitable contradiction argument.

From Theorem \ref{thm:ReifregX} we could immediately appeal on the intrinsic Reifenberg theorem of Cheeger and Colding \cite[Theorem A.1.1]{ChCo97} and get the conclusion of Theorem \ref{th:Holderreg}.  We prefer to provide an explicit construction of a bi-Hölder homeomorphism obtained from almost splitting maps through a Transformation Theorem, in the spirit of \cite{CJN}.  One key new point in our approach is an almost-rigidity statement implying that for sufficiently small $\delta$, if a point $x$ in a non-collapsed strong Kato limit space satisfies $$\uptheta(t,x)<1+\delta,$$ then an almost splitting map realizing a GH-isometry exists from $B_{\sqrt{t}}(x)$ to an Euclidean ball of radius $\sqrt{t}$.  We next prove a Transformation Theorem that eventually provides a better regularity on such harmonic maps : these are bi-Hölder homeomorphisms. The proof of the Transformation Theorem is a direct one and uses some results of \cite{CMT} about convergence of harmonic functions together with the refinements that we develop in Section 3.

\medskip

\noindent \textbf{Acknowledgments:}
The authors are partially supported by the ANR grant ANR-17-CE40-0034: CCEM. The first and third authors thank the Centre Henri Lebesgue ANR-11-LABX-0020-01 for creating an attractive mathematical environment. The first author is also partially supported by the ANR grant ANR-18-CE40-0012: RAGE.

\section{Preliminaries}

In a metric space $(X, \dist)$ we denote by $B_r(x)$ the open ball of radius $r$ centered at $x \in X$. Letting $B=B_r(x)$, for any $\lambda>0$ we denote by $\lambda B$ the re-scaled ball centered at $x$ of radius $\lambda r$. We call metric measure space any triple $(X,\dist,\mu)$ where $(X,\dist)$ is a geodesic and proper metric space and $\mu$ is a fully supported Borel measure such that $\mu(B_r(x))$ is strictly positive and finite for any $x \in X$ and $r>0$. 

The Cheeger energy of $(X,\dist,\mu)$ is the map $\Ch\colon \Lip_c(X)\rightarrow \R_+$ defined by
$$\Ch(f)=\int_X \text{lip}^2(f)\di\mu,$$
where $\text{lip}(f)$ denotes the local Lipschitz constant of $f$. Following \cite{GigliMAMS,Gigli} we say that $(X,\dist,\mu)$ is infinitesimally Hilbertian if $\Ch$ is quadratic, in which case the closure of $\Ch$, still denoted by $\Ch$,  is a Dirichlet form with domain denoted by $H^{1,2}(X,\dist,\mu)$.  We write $L$ for the associated non-positive, self-adjoint operator and $\{e^{-tL}\}_{t>0}$ for the Markov semi-group generated by $L$. For any $f\in H^{1,2}(X,\dist,\mu)$ there exists a unique $| df| \in L^2(X,\mu)$ called minimal relaxed slope of $f$ such that
$$\Ch(f)=\int_X |df|^2\di\mu.$$ Moreover, $\Ch$ is strongly local and regular,  and its carré du champ is given by
\begin{align*}
\di \Gamma(u,v) = \frac{1}{4} (|d(u+v)|^2 - |d(u-v)|^2) \di \mu =: \langle d u ,d v \rangle \di \mu 
\end{align*}
for any $u,v \in H^{1,2}(X,\dist,\mu)$. For any open set $\Omega \subset X$ we also set 
\[
H^{1,2}_{loc}(\Omega,\dist,\mu):=\{f \in L^2_{loc}(\Omega,\mu) \, : \,  \phi f \in H^{1,2}(X,\dist,\mu)  \,\, \text{for any $\phi \in \Lip_c(\Omega)$} \}.
\]
We say that $f \in H^{1,2}_{loc}(\Omega,\dist,\mu)$ is harmonic in $\Omega$ if for any $\phi \in \Lip_c(\Omega)$,
\[
\int_{\Omega} \langle d f, d \phi \rangle \di \mu  = 0.
\]

If $(M^n,g)$ is a smooth and connected Riemannian manifold, the Cheeger energy of $(M,\dist_g,\nu_g)$ coincides with its usual Dirichlet energy.  We often implicitly identify a Riemannian manifold $(M^n,g)$ with its isometry class or with the metric measure space $(M,\dist_g,\nu_g)$.  

For any positive integer $k$, we denote by $\setB_r^k$ the Euclidean ball of radius $r$ centered at the origin of $\setR^k$, and we write $\setB_r^k(p)=p+\setB_r^k$ for any $p \in \setR^k$. 

\subsection{Notions of convergence}  We assume the reader to be familiar with the various notions of Gromov-Hausdorff convergence; we refer to \cite[Section 11]{HKST}, for instance, if this is not the case.  We simply recall that a map \mbox{$\phi : (X,\dist_X) \to (Y,\dist_Y)$} is called an $\eps$-GH isometry if $|\dist_X(x,x') - \dist_Y(\phi(x),\phi(x'))|\le \eps$ for any $x,x' \in X$ and for any $y \in Y$ there exists $x \in X$ such that $\dist_Y(\phi(x),y)\le \eps$. If $\{(X_\alpha,\dist_\alpha,o_\alpha)\}, (X,\dist,o)$ are pointed metric spaces such that ${(X_\alpha,\dist_\alpha,o_\alpha)\to (X,\dist,o)}$ in the pointed Gromov-Hausdorff topology, we denote by $x_\alpha \in X_\alpha \to x \in X$ a convergent sequence of points, following \cite[Characterization 1]{CMT} and the definition soon after. 

\subsubsection{Tangent cones} Let us recall the classical definitions of tangent cones.

\begin{D}\label{D:tangent}
\begin{enumerate}
\item Let $(X, \dist)$ be a metric space.  For any $x \in X$,  we call metric tangent cone of $(X,\dist)$ at $x$ any pointed metric space $(Y,\dist_Y,x)$ obtained as a limit point in the pointed Gromov-Hausdorff topology of the family of rescalings $\{(X,r^{-1}\dist,x)\}_{r>0}$ as $r\downarrow 0$. We denote by $\Tan(X,x)$ the set of metric tangent cones of $(X,\dist)$ at $x$.
\item  Let $(X, \dist,\mu)$ be a metric measure space.  For any $x \in \supp \mu$,  we call metric measured tangent cone of $(X,\dist,\mu)$ at $x$ any pointed metric measure space $(Y,\dist_Y,\mu_Y,x)$ obtained as a limit point in the pointed measured Gromov-Hausdorff topology of the family of rescalings $\{(X,r^{-1}\dist,\mu(B_r(x))^{-1}\mu, x)\}_{r>0}$ as $r\downarrow 0$. We denote by $\Tan_\meas(X,x)$ the set of metric measured tangent cones of $(X,\dist, \mu)$ at $x$.
\end{enumerate}
\end{D}

We are especially interested in tangent cones which split off a Euclidean factor. Let us recall the definition. 

\begin{D} Let $k$ be a positive integer.
\begin{enumerate}
\item We say that a pointed metric space $(X, \dist,x)$ splits off an $\setR^k$ factor if there exists a pointed metric space $(Z,\dist_Z,z)$ and an isometry $\phi : X \to \setR^k \times Z$ such that $\phi(x)=(0,z)$.

\item  We say that a pointed metric measure space $(X, \dist,\mu,o)$ splits off an $\setR^k$ factor if there exists a pointed metric measure space $(Z,\dist_Z,\mu_Z,z)$ and an isometry $\phi : X \to \setR^k \times Z$ such that $\phi(x)=(0,z)$ and $\phi_\# \mu =  \mathcal{H}^k \otimes \mu_Z$.
\end{enumerate}
Here and in the sequel the space $\setR^k \times Z$ is implicitly equipped with the classical Pythagorean product distance.
\end{D}

\subsubsection{Convergence of functions} 
 Let us recall now some notions of convergence for functions defined on varying spaces.  We refer to \cite[Section 1.4]{CMT} and the references therein for a more exhaustive presentation.
\begin{D}
Let $\{(X_\alpha, \dist_\alpha, \mu_\alpha, o_\alpha)\}_\alpha, (X,\dist,\mu,o)$ be infinitesimally Hilbertian metric measure spaces such that $(X_\alpha, \dist_\alpha, \mu_\alpha, o_\alpha) \to (X,\dist,\mu,o)$ in the pointed measured Gromov-Hausdorff topology. 

\begin{enumerate}
\item Let $\varphi_\alpha  \in \cC_c(X_\alpha)$ for any $\alpha$ and $ \varphi\in \cC_c(X)$ be given.  We say that $\left\{\varphi_\alpha\right\}$ converges uniformly on compact sets to $\varphi$ if $\varphi_\alpha(x_\alpha) \to \varphi(x)$ whenever $x_\alpha \in X_\alpha \to x \in X$; we write $\varphi_\alpha\stackrel{\cC_c}{\longrightarrow} \varphi $ if this convergence holds.

\hfill

\item Let $f_\alpha \in L^2(X_\alpha,\mu_\alpha)$ for any $\alpha$ and $f \in L^2(X,\mu)$ be given.
\begin{itemize}
\item We say that $\{f_\alpha\}$ converges to $f$ weakly in $L^2$ if $\sup_\alpha \|f_\alpha\|_{L^2}<+\infty$ and
$$\int_{X_\alpha}\varphi_\alpha f_\alpha\di\mu_\alpha\to \int_{X}\varphi f\di\mu$$
whenever $\varphi_\alpha\stackrel{\cC_c}{\longrightarrow} \varphi $; we write $f_\alpha \stackrel{L^2}{\weakto}f$ if this convergence holds.
\item We say that $\{f_\alpha\}$ converges to $f$ strongly in $L^2$ if $f_\alpha \stackrel{L^2}{\weakto}f$ and  $\lim_\alpha  \|f_\alpha\|_{L^2}= \|f\|_{L^2}$; we write $f_\alpha \stackrel{L^2}{\to}f$ if this convergence holds.
\end{itemize}
\hfill

\item Let $f_\alpha \in H^{1,2}(X_\alpha,\dist_\alpha,\mu_\alpha)$ for any $\alpha$ and $f \in H^{1,2}(X,\dist,\mu)$ be given.
\begin{itemize}
\item We say that $\{f_\alpha\}$ converges to $f$ weakly in energy if 
$f_\alpha \stackrel{L^2}{\weakto}f$ and $\sup_\alpha \Ch_\alpha(f_\alpha)<+\infty$; we write $f_\alpha \stackrel{\mE}{\weakto}f$ if this convergence holds.
\item We say that $\{f_\alpha\}$ converges to $f$ strongly in energy if $f_\alpha \stackrel{\mE}{\weakto}f$ and $\lim_\alpha \Ch_\alpha(f_\alpha)=\Ch(f)$; we write $f_\alpha \stackrel{\mE}{\to}f$ if this convergence holds.
\end{itemize}

\hfill

\end{enumerate}
\end{D}
\subsection{Kato bound and Kato limits} Recall that $T$, $f$ are fixed and satisfy \eqref{eq:f}.  
The following has been proved in \cite[Proposition 2.3]{CMT}.

\begin{prop}\label{eq:PI_manifolds} 
There exists $\upkappa\ge 1$ and $\uplambda>0$ depending only on $n$ such that any $(M^n,g) \in \mathcal{K}(n,f)$ satisfies
\begin{itemize}
\item[1.] a uniform volume estimate: for any $x\in M$ and $0<s<r\le \sqrt{T}$, 
\begin{equation}\label{eq:volestim}\frac{\nu_g(B_r(x))}{\nu_g(B_s(x))}\le \upkappa \left(\frac{r}{s}\right)^{e^2n},\end{equation}
\item[2.] a uniform local Poincaré inequality: for any ball $B\subset M$ with radius $r\le \sqrt{T}$  and any $\varphi\in \cC^1(B)$,
\begin{equation}\label{eq:Poincaré}\int_B \left(\varphi-\fint_B \varphi\di \nu_g\right)^2 \di \nu_g\le \uplambda r^2 \int_B |d\varphi|^2 \di \nu_g.\end{equation}
\end{itemize}
\end{prop}

\begin{rem}
Note that \eqref{eq:volestim} implies a so-called doubling condition:
\begin{equation}\label{eq:doubling}
\nu_g(B_{2r}(x)) \le A(n) \nu_g(B_r(x))
\end{equation}
for any $x \in X$ and $r \in(0, \sqrt{T}/2]$, where $A(n):=\upkappa 2^{e^2n}$.  We shall often use the following consequence of the doubling condition: for any $\lambda\in(0,1)$ there exists a constant $C(n,\lambda)\ge 1$ such that for any ball $B\subset M$ and any locally integrable ${\phi : B \to \setR}$, 
\begin{equation}\label{eq:cor_doublement}
\fint_{\lambda B}|\phi| \di \nu_g \leq C(n,\lambda) \fint_{B}|\phi| \di \nu_g.
\end{equation}
\end{rem}

The next proposition collects estimates on the heat kernel of $(M^n,g) \in \mathcal{K}(n,f)$. 

\begin{prop}\label{Prop:heatKe} There exists a constant $\upgamma\ge 1$ depending only on $n$ such that for any $(M^n,g) \in \mathcal{K}(n,f)$, for all $x,y\in M$ and $t\in (0,T)$,
\begin{enumerate}[i)]
\item 
$\displaystyle \frac{\upgamma^{-1}}{\nu_g(B_{\sqrt{t}}(x))}e^{-\upgamma\frac{\dist^2(x,y)}{t}}\le H(t,x,y)\le \frac{\upgamma}{\nu_g(B_{\sqrt{t}}(x))}e^{-\frac{\dist^2(x,y)}{5t}},$
\item $\displaystyle \left|\frac{\partial}{\partial t} H(t,x,y)\right|\le  \frac{\upgamma}{t \nu_g(B_{\sqrt{t}}(x))}e^{- \frac{\dist^2(x,y)}{5t}},$
\item $\displaystyle  \left|d_x H(t,x,y)\right|\le  \frac{\upgamma}{\sqrt{t}\nu_g(B_{\sqrt{t}}(x))}e^{-\frac{\dist^2(x,y)}{5t}}.$
\end{enumerate}
\end{prop}
\begin{proof} 
The first estimate i) was established in \cite{C16}, see also \cite[Proposition 3.3]{CMT}.  The second estimate ii) is a consequence of i), see e.g.~\cite[Corollary 3.1]{GriderHeat}. The third estimate iii) is a consequence of the Li-Yau inequality \cite[Proposition 3.3]{C16}:
$$e^{-2}\left|d_x H(t,x,y)\right|^2\le \frac{e^2n}{2t} H^2(t,x,y)+H(t,x,y)\left|\frac{\partial}{\partial t} H(t,x,y)\right|,$$
together with i) and ii).
\end{proof}

Let us now recall a couple of results from \cite{CMT} about Kato limit spaces.

\begin{prop}
Any $(X,\dist,\mu,o) \in \katolimits$ is an infinitesimally Hilbertian space satisfying the doubling condition \eqref{eq:doubling} and the local Poincaré inequality \eqref{eq:Poincaré}. Moreover, for any $x \in X$, any $(Y,\dist_Y,\mu_Y,x) \in \Tan_\meas(X,x)$ is a pointed $\RCD(0,n)$ space.
\end{prop}

Metric measure spaces satisfying an $\RCD(0,n)$ bound have, in a synthetic sense, non-negative Ricci curvature and dimension less than $n$. We refer to  \cite{GigliLectureNotes} for a survey about their properties.  

From \cite{CMT}, we also know that the following hold.

\begin{prop}
\label{prop:HKconv}
Let $\{(M^n_\alpha, \dist_\alpha, \mu_\alpha,o_\alpha)\} \subset \cK_\meas(n,f,c)$ be converging to $(X,\dist, \mu,o)\in \overline{\cK_\meas(n,f,c)}$ in the pointed measured Gromov-Hausdorff topology. Let $H_\alpha$ be the heat kernel of $(M_\alpha^n,g_\alpha)$ for any $\alpha$. Then $X$ admits a locally Lipschitz heat kernel, that is to say a map $H:(0,+\infty)\times  X \times X \to (0,+\infty)$ such that 
$$\left(e^{-tL}f\right)(x)=\int_X H(t,x,y)f(y)\di\mu(y)$$for any $f\in L^2(X,\mu)$, any $t>0$ and $\mu$-a.e.~$x\in X$. Moreover,  $H$ satisfies the three estimates in Proposition \ref{Prop:heatKe}. Furthermore, the following convergence results hold.
\begin{itemize}
\item For any $t>0$ and $x_\alpha\in M_\alpha\to x\in X$, $y_\alpha\in M_\alpha\to y\in X$,
\begin{equation}\label{CvHeat}
H_\alpha(t,x_\alpha,y_\alpha) \to H(t,x,y) \quad \text{and} \quad \frac{\partial}{\partial t}H_\alpha(t,x_\alpha,y_\alpha) \to \frac{\partial}{\partial t}H(t,x,y).
\end{equation}
\item For any  $t>0$ and $x_\alpha \in M_\alpha \to x \in X$,
\begin{equation}\label{L2CvHeat}
H_\alpha(t,x_\alpha,\cdot) \stackrel{L^2}{\rightarrow} H(t,x,\cdot).
\end{equation}
\end{itemize}
\end{prop}

As an important consequence, we derive in the next statement a Li-Yau inequality for Kato limit spaces. 

\begin{prop}\label{lem:LY} Consider $(X,\dist, \mu,o)\in \overline{\cK_\meas(n,f,c)}$.  Set $\gamma(t)=\exp\left(8\sqrt{nf(t)}\right)$ for any $t\in(0,T]$. Then for any $x\in X$ and $t\in(0,T]$,  the Li-Yau inequality
\begin{equation}\label{LYX} 
\gamma^{-1}(t)\left| d H(t,x,\cdot )\right|^2-H(t,x,\cdot)\frac{\partial}{\partial t}H(t,x,\cdot)\le \frac{n\gamma(t)}{2t} H^2(t,x,\cdot)\end{equation}
holds $\mu$-a.e.~on $X$.\end{prop}
\proof Let $\{(M^n_\alpha, \dist_\alpha, \mu_\alpha,o_\alpha)\} \subset \cK_\meas(n,f,c)$ be converging to $(X,\dist, \mu,o)$ in the pointed measured Gromov-Hausdorff topology.  By \cite[Proposition 3.3]{C16}, for any $x,y\in M_\alpha$ and $t\in(0,T]$,
\begin{equation}\label{eq:LYXa}
\gamma^{-1}(t)\left|d_yH_\alpha(t,x,y)\right|^2-H_\alpha(t,x,y)\frac{\partial}{\partial t}H_\alpha(t,x,y)\le \frac{n\gamma(t)}{2t} H_\alpha^2(t,x,y).
\end{equation}
Take $x_\alpha\in M_\alpha\to x\in X$ and set $u_\alpha(y)=H_\alpha(t,x_\alpha,y)$ for any $y \in M_\alpha$ and any $\alpha$.  The $L^2$ heat kernel convergence \eqref{L2CvHeat} yields $$u_\alpha\stackrel{L^2}{\rightarrow} u:=H(t,x,\cdot).$$
Moreover,  the semi-group property implies
\begin{equation}\label{eq:semigroup}
\int_{M_\alpha} |du_\alpha|^2\di\nu_{g_\alpha}=\int_{M_\alpha} u_\alpha\Delta_{g_\alpha}u_\alpha\di\nu_{g_\alpha}=-\frac{1}{2}\frac{\partial}{\partial t}H_\alpha(2t,x_\alpha,x_\alpha)=-\frac{\partial H_\alpha}{\partial t}(2t,x_\alpha,x_\alpha),
\end{equation}
hence by Proposition \ref{Prop:heatKe}.ii) the sequence $\{u_\alpha\}$ is bounded in energy, hence $u_\alpha\stackrel{E}{\weakto} u$ by definition. Since the semi-group property also implies \eqref{eq:semigroup} with $u$,  $H$ and $x$ in place of $u_\alpha$, $H_\alpha$ and $x_\alpha$ respectively,  the convergence \eqref{CvHeat} yields $\lim_{\alpha} \| du_\alpha\|_{L^2}=\Ch(u),$ hence by definition $u_\alpha\stackrel{E}{\rightarrow} u.$ Proposition \ref{Prop:heatKe}.iii)  implies that the sequence $\{|du_\alpha|\}$ is locally bounded in $L^\infty$ hence
 with \cite[Proposition E.7]{CMT} we can conclude that
$$|du_\alpha|\stackrel{L^2}{\weakto} |du|.$$ 
This convergence,  together with \eqref{CvHeat} and \eqref{eq:LYXa},  implies \eqref{LYX}.
\endproof

\begin{rem}\label{rem:mieux} If there exists $\tau \in (0,T]$ such that
$$\lim_{\alpha} \mbox{k}_\tau(M_\alpha,g_\alpha)=0,$$ 
then for any $x\in X$ and $t\in (0,\tau]$,  the Li-Yau inequality
$$\left|dH(t,x,\cdot)\right|^2-H(t,x,\cdot)\frac{\partial}{\partial t}H(t,x,\cdot)\le \frac{n}{2t} H^2(t,x,\cdot)$$ holds $\mu$-a.e.~on $X$.
\end{rem}

\section{Almost splittings maps and consequences of GH-closeness on functions}

In this section, we define $(k,\eps)$-splitting maps on Kato limits and prove some relevant properties.  Such maps were introduced in \cite{ChCo96,C97,ChCo97} for the study of Ricci limit spaces and extensively used later in the study of limit spaces and $\RCD(K,N)$ spaces, see for instance \cite{CheegerNaber,CJN,Bamler,BPS}. 

From now on, for any positive integer $k$, we let $\cM_k(\setR)$ be the space of $k\times k$ matrices with real entries, $\cS_k(\setR)\subset \cM_k(\setR)$ be the subspace made of symmetric matrices, and we denote by $\|\cdot\|$ the matrix norm induced by the Euclidean norm $|\cdot|$, meaning that $\|A\|^2:=\sup\{ ^t(A \xi)A \xi  \, : \, \text{$\xi \in \setR^k$ such that $^t\xi \xi=1$}\}$ for any $A \in \cM_k(\setR)$.  Then the following holds.

\begin{lemma}\label{lem:GS}
Assume that $A \in \cS_k(\setR)$ is positive definite. Then there exists a unique lower triangular matrix $T \in \cM_k(\R)$ such that 
\begin{equation}\label{eq:prox1}T A {}^t T = \Id_k.\end{equation}
Moreover, if there exists $\eps \in (0,1/2)$ such that $A \in \cS_k(\setR)$ satisfies
\begin{equation}\label{eq:prox}
\|A-\Id_k\|< \eps, 
\end{equation}
then for some $C_k$ depending only on $k$, the matrix $T$ satisfies
\begin{equation}\label{eq:prox3}
\| T-\Id_k\| < C_k \eps.
\end{equation}\end{lemma}

\begin{rem}
The matrix ${}^tT$ is obtained by applying the Gram-Schmidt process. \end{rem}

\subsection{Almost splitting maps} For any infinitesimally Hilbertian metric measure space $(X,\dist,\mu)$, whenever a map $u=(u_1,\ldots,u_k) : B \to \setR^k$ satisfies $u_i \in H^{1,2}(B,\dist,\mu)$ for any $i \in \{1,\ldots,k\}$ we define the Gram matrix map of $u$ as the $\cS_k(\setR)$-valued map
\[
G_u:=[ G_{i,j}] \qquad \text{where $G_{i,j} := \langle d u_i, d u_j\rangle$ for any $1\le i,j \le k$,}
\]
and we set $|d G_u|^2:=\sum_{1\le i,j \le k} | d G_{i,j}|^2$. Note that if $T$ is a lower triangular $k \times k$ matrix and $\tilde{u}:=T\circ u$, then
\begin{equation}\label{eq:gramtilde}
G_{\tilde u} = T G_u {}^tT \qquad \text{$\mu$-a.e.~in $B$.}
\end{equation}

\begin{D}
Let $(X,\dist,\mu,o) \in \overline{\mathcal{K}_\meas(n,f,c)}$. Let $B \subset X$ be a ball of radius $r>0$, $k \in \{1,\ldots,n\}$ and $\eps>0$.
\begin{enumerate}
\item We call $(k,\eps)$-splitting of $B$ any harmonic map $u : B \to \setR^k$ such that $\|du\|_{L^\infty(B)}\le 2$ and
\begin{equation}\label{eq:split}
\fint_B \|G_u - \Id_k\| \di \mu < \eps.
\end{equation}
\item We say that a $(k,\eps)$-splitting $u$ of $B$ is reinforced if 
\[ \fint_B (\|G_u - \Id_k\| + r^2 |d G_u|^2) \di \mu < \eps.\]
\item We say that a (possibly reinforced) $(k,\eps)$-splitting $u$ of $B$ is balanced if
\[
\fint_B G_u \di \mu = \Id_k.
\]
\end{enumerate}
\end{D}
\begin{rem}
Assumption $\|du\|_{L^\infty(B)}\le 2$ implies
\[ \sup_{1 \le i,j \le k} |G_{i,j}(y)|\le 4 \qquad \text{for $\mu$-a.e.~$y \in B$}.\]
\end{rem}

\begin{rem}\label{rem:GSsplitting}
Condition \eqref{eq:split} implies that the symmetric matrix $A_u=\fint_{B}G_u\di \mu$ is \mbox{$\eps$-close} to the identity $\Id_k$. As a consequence of Lemma \ref{lem:GS} applied with $A=A_u$, for any $\eps \in (0,1/2)$  and any \mbox{$(k,\eps)$-splitting} $u: B \to \R^k$ there exists a lower triangular matrix $T$ with $\|T\|\leq 1+C_k\eps$ such that the map \mbox{$\tilde u = T \circ u: B\to \R^k$} satisfies
\begin{equation}\label{eq:GSsplitting}\fint_{B} G_{\tilde u}\di \mu= \Id_k \qquad \text{and} \qquad \fint_{B} \|G_{\tilde u}- \Id_k \|\di \mu < (1+C_k\eps)^2\eps.\end{equation}

\end{rem}

\begin{rem}\label{rem:reinforced}
The definition of reinforced splitting is just a technical convenience. Indeed,  by means of Bochner's formula and of the Hessian bound given in \cite[Proposition 3.5]{CMT}, one can prove that any splitting on a Riemannian manifold with a Kato bound is a reinforced splitting on a ball with smaller radius, and then show that this property for manifolds with a uniform Kato bound is stable under pointed measured Gromov-Hausdorff convergence. This implies, in particular, that if $u$ is a reinforced splitting of a ball $B$ in a Kato limit space, then the coefficients of the Gram matrix map $G_u$ all belong to $H^{1,2}_{loc}(B,\dist,\mu)$.
\end{rem}

The next result provides an improvement of the local Lipschitz constant for splittings.

\begin{prop}
\label{Prop:LipMieux}
Let $(M^n,g)$ be a closed Riemaniann manifold,  $B \subset M$ a ball of radius $r>0$, $k \in \{1,\ldots,n\}$, $\eta \in (0,1)$, $L>1$ and $u:B\to \R^k$ a harmonic map such that $\|du\|_{L^\infty(B)}\le L$.  Let $G_u$ be the Gram matrix map of $u$. Assume that there exists $\delta \in (0,1/16n]$ such that
$$\kato_{r^2}(M^n,g) < \delta, \quad \fint_{B}\| G_u -\Id_k\| \di \vol_g < \delta.$$
Then  there exists $C(n,\eta,L)>0$ such that $\|du\|_{L^\infty(\eta B)}\le 1+C(n,\eta,L)\delta$.
\end{prop}

\begin{proof}
In the proof of \cite[Proposition 7.5]{CMT}, use the gradient bound iii) in Proposition \ref{Prop:heatKe} to get $II \le C \delta$ instead of $II \le C \delta^{1/2}$. Apply the resulting statement to any  function $u_\xi := \langle \xi,u\rangle$ with $\xi \in \setR^k$ satisfying $|\xi|=1$, and conclude by taking $\xi=du/|du|$ pointwise.
\end{proof}

\subsection{GH-closeness and harmonic functions} In the setting of uniform lower Ricci bounds, existence of almost splitting maps is closely related to mGH-closeness of a ball to a Euclidean ball. We show below that the same relation actually holds for Kato limit spaces.

Thoughout this subsection, we let $k \in \{1,\ldots,n\}$ be fixed. We denote by $\|\cdot\|_1$ the $L_{1,1}$ matrix norm, namely $\|M\|_1=\sum_{i,j=1}^k |m_{i,j}|$ for any $M \in \cM_k(\setR)$. Note that $\|\cdot\|\leq \|\cdot\|_1$.

\begin{theorem}\label{MetaThm}
For all $\eps, \eta,\lambda \in (0,1)$ such that $\lambda < \eta$ there exists $\nu$ depending only on $\eps, \eta, \lambda,n, f,c$ such that if $(X,\dist, \mu, o), (X', \dist', \mu',o') \in \overline{\cK_\meas(n,f,c)}$, $x \in X$, $x' \in X'$ and $r \in (0, \sqrt{T}],$ are such that 
$$\dmGH(B_r(x),B_r(x'))< \nu r,$$
if $h: B_r(x) \to \R^k$ is a harmonic function satisfying $\|dh\|_{L^\infty(B_r(x))}\le L$ for some $L>1$, then there exists a harmonic function $h': B_{\eta r}(x') \to \R^k$ satisfying $\|dh'\|_{L^\infty(B_{\eta r}(x'))}\le L C(n,\eta)$ for some $C(n,\eta)\geq 1$ and:
\begin{enumerate}
\item $\|h'\circ \Phi - h\|_{L^{\infty}(B_{\eta r}(x))}< \eps r,$ where $\Phi$ is a $(\nu r)$-GH isometry between $B_r(x)$ and $B_r(x')$;
\item for all $s \in [\lambda r,\eta r]$
$$\left\| \fint_{B_s(x)} G_h \di\mu - \fint_{B_s(x')}G_{h'}\di\mu' \right\| < \eps,$$
\item for all $A\in \cM_k(\R^k)$ and $s \in [\lambda r,\eta r]$, 
\[
\left| \fint_{B_{s}(x)}\|G_h-A\|_1 \di \mu - \fint_{B_{s}(x')}\|G_{h'}-A\|_1\di \mu'\right| \le  \eps.
\]

\end{enumerate}
\end{theorem}

The previous is a consequence of the analysis made in \cite[Appendix A]{CMT}.  For the sake of completeness, we provide a proof in Appendix \ref{app:meta}.

Theorem \ref{MetaThm} has the following direct consequence about existence of reinforced almost splittings.

\begin{prop}
\label{prop:ExiSplit}
For any $\eps, \eta \in (0,1)$ there exists $\delta >0$ depending on $n,f,c, \eps$ and $\eta$ such that if $(X,\dist, \mu,o)\in \katolimits$, $x\in X$ and  $r \in (0,\sqrt{T}]$ satisfy
$$\dmGH(B_r(x),\mathbb{B}_r^k)< \delta r,$$
then there exists a reinforced $(k,\eps)$-splitting of $B_{\eta r}(x)$. 
\end{prop}

Moreover,  Theorem \ref{MetaThm}  implies that almost splittings are GH-isometries under the appopriate assumptions. 

\begin{prop}
\label{prop:GHisometry}
For any $\eps, \eta \in (0,1)$ there exist $\delta >0$ depending on $n, f, c, \eps$ and $\eta$ and a constant $C(n,\eta)>0$, such that for all $(X,\dist, \mu, o) \in \katolimits$, if $u: B_r(x) \to \R^k$ is a $(k,\eps)$-splitting and 
$$\dmGH(B_r(x), \mathbb{B}^k_r) < \delta r,$$
then $u$ is a $\left(C(n,\eta)\sqrt{\eps}r\right)$-GH isometry between $B_{\eta r}(x)$ and $\mathbb{B}_{\eta r}^k(u(x))$. 
\end{prop}

The proof of this proposition relies on the following Euclidean result. 
\begin{lemma} 
\label{lem:harmGH}
If $v\colon \mathbb{B}^k \to \R^k$ is a harmonic map such that 
$$\fint_{\mathbb{B}^k} \| G_v-\Id_k\|\le \eps,$$
then $v\colon \mathbb{B}^k_{\eta} \to \R^k$ is a $\left(C(n,\eta)\sqrt{\eps}\right)$-GH isometry  between $\mathbb{B}_{\eta }$ and $\mathbb{B}_{\eta }^k(v(0))$.
\end{lemma}

\begin{proof} We will assume that $\eta\ge 1/2$. Consider a cut-off function $\chi$ equal to $1$ on $\frac{1+\eta}{2}\setB^k$ and vanishing outside $\frac{3+\eta}{4}\setB^k$, with 
$$\|\Delta \chi\|_{L^\infty}\le C(k,\eta).$$
By the Bochner formula we have:
$$|\Hess v|^2+\frac12 \Delta( G_v-\Id_k)=0,$$
then multiplying by $\chi$ and integrating by parts lead to the estimate:
$$\int_{\frac{1+\eta}{2}\setB^k} |\Hess v|^2\le C(k,\eta) \int_{\setB^n} \left\| G_v-\Id_k\right\|\le C(k,\eta) \eps $$

Using classical elliptic estimate, we obtain a $\cC^2$ estimate on $v$:
$$\|\Hess v\|_{L^{\infty}(\eta\setB^k)}\le C(k,\eta)\sqrt{\eps}.$$
With Taylor formula, we get that for any $x\in \eta\setB^k$,
$$|v(x)-v(0)-dv(0)(x)|\le C(k,\eta)\sqrt{\eps} \, \, \, \, \text{ and } \, \, \, \, |dv(0)-dv(x)|\le C(k,\eta)\sqrt{\eps}.$$
But we also have
$$\fint_{\eta\mathbb{B}^k} \| G_v-\Id_k\|\le \eta^{-k}\fint_{\mathbb{B}^k} \| G_v-\Id_k\|\le 2^k \eps.$$
Hence we find a point $x_o\in\eta \setB^k$ such that 
$$ \| G_v(x_o)-\Id_k\|\le 2^k \eps.$$
Using the polar decomposition of $dv(x_o)$ we obtain a linear isometry $g\in \text{O}(k)$ such 
$$|dv(x_o)-g|\le C(k) \eps.$$
Introducing the affine isometry $\iota:=v(0)+g$ we get that for any $x\in \eta\setB^k$,
$$|v(x)-\iota(x)|\le C(k,\eta)\sqrt{\eps}.$$
Setting $C'(n,\eta)= \max_{1\le k \le n} C(k,\eta)$ eventually leads to the desired result.
\end{proof}
\begin{proof}[Proof of Proposition \ref{prop:GHisometry}] We let $\eps, \eta \in (0,1)$. We will assume that $\eta\ge 1/2$.  With Theorem \ref{MetaThm}, we find some $\delta(n,\eps,\eta,f,c)>0$ such that  if $(X,\dist, \mu, o) \in \katolimits$, if $u: B_r(x) \to \R^k$ is a $(k,\eps)$-splitting and $$\dmGH(B_r(x), \mathbb{B}^k_r) < \delta r,$$ then there is some harmonic map
$$v\colon \setB^k_{(1+\eta)\frac{r}{2}}\to \R^k$$ and some $\delta r$-GH isometry
$\Phi\colon B_r(x)\to\setB^k_{r}$ such that 
\begin{equation}\label{approxGH}
\|v\circ\Phi-u\|_{L^\infty(B_{(1+\eta)\frac{r}{2}}(x))}\le \eps r\end{equation} and
$$\left|\fint_{B_{(1+\eta)\frac{r}{2}}(x)} \|G_u-\Id_k\|_1\di\mu-\fint_{\setB^k_{(1+\eta)\frac{r}{2}}} \|G_v-\Id_k\|_1\right|\le \eps.$$
Observe that the doubling condition and the equivalence of the norms $\|\cdot \|$ and $\|\cdot\|_1$ yield
$$\fint_{B_{(1+\eta)\frac{r}{2}}(x)} \|G_u-\Id_k\|_1\di\mu\le  A(n) \fint_{B_{r}(x)} \|G_u-\Id_k\|_1\di\mu\le C(n)\eps$$
for some $C(n)$ only depending on $n$. Since $\|\cdot \|\leq \|\cdot \|_1$, we get 
$$\fint_{\setB^k_{(1+\eta)\frac{r}{2}}} \|G_v-\Id_k\|\le (1+C(n))\eps.$$ Hence according to the previous lemma, we know that $v$ is a $\left(C(n,\eta)\sqrt{\eps}r\right)$-GH isometry between $\setB^k_{\eta r}$ and itself. Using \eqref{approxGH}, we obtain the desired conclusion about the restriction of $u$ to $B_{\eta r}(x)$.
 \end{proof}

\subsection{Propagation of reinforced almost splittings} The next result is an important propagation property of reinforced splittings.  

\begin{prop}[Propagation of reinforced splittings]\label{prop:propagation} Consider $(X,\dist,\mu,o) \in \overline{\mathcal{K}_\meas(n,f,c)}$.  There exists $C>0$ depending only on $n$ such that for any $k \in \{1,\ldots,n\}$ and $\eps\in(0,1)$, if $u$ is a reinforced $(k,\eps)$-splitting of a ball $B_r(x) \subset X$ with $r \in (0,\sqrt{T})$, then there exists a Borel set $\Omega_\eps \subset B_{r/2}(x)$ such that:
 \begin{enumerate}
 \item[(A)]\label{A} $\mu(B_{r/2}(x) \backslash \Omega_\eps) \le C\sqrt{\eps}\mu(B_{r/2}(x))$,
 \item[(B)]\label{B} the restriction of $u$ to $B_s(y)$ is a reinforced $(k,\sqrt{\eps})$-splitting for any $y \in \Omega_\eps$ and $s \in (0,r/2)$,
 \item[(C)]\label{C}   for $\mu$-a.e.~$y \in \Omega_\eps$,  for any $\xi \in \R^k$,
\begin{equation}\label{eq:bilipGram}
(1-\sqrt{\eps})|\xi|^2\le {}^t \xi G_u(y)\xi\le (1+\sqrt{\eps})|\xi|^2,
\end{equation}

 \item[(D)]\label{D} any $y \in \Omega_\eps$ is such that any $(Y,\dist_Y,\mu_Y,y) \in \mathrm{Tan}_\meas(X,y)$ splits off a factor $\setR^k$.
 \end{enumerate}
 \end{prop}

\begin{proof}Let $x \in X$ and $r \in (0,\sqrt{T})$. Assume that $u:B_r(x)\to \setR^k$ is a $(k,\eps)$-splitting.  Set \[\Omega_\eps:=\{y \in B_{r/2}(x)\, : \, M_{r/2}v(y) \le \sqrt{\eps}\}\]
 where \[v:=\|G_u - \Id_k\| + r^2|d G_u|^2\]
 and
 \[
 M_{r/2}v(y):= \sup_{s \in (0,r/2)} \fint_{B_s(y)} v \di \mu.
 \]
The definition of $\Omega_\eps$ is made so that (B) is satisfied. Let us prove (A). For any $y \in B_{r/2}(x)\backslash \Omega_{\eps}$ there exists $s_y \in (0,r/2)$ such that $\mu(B_{s_{y}}(y)) < (\sqrt{\eps})^{-1} \int_{B_s(y)} v \di \mu$.  By the Vitali covering lemma, there exists a countable family of points $\{y_i\} \subset B_{r/2}(x)\backslash \Omega_\eps$ such that the balls $\{B_{s_{y_i}}(y_i)\}$ are pairwise disjoint and $B_{r/2}(x)\backslash \Omega_\eps \subset \bigcup_i B_{5s_{y_i}}(y_i)$. Then, with a constant $C$ depending only on $n$ which may change from line to line,
 \begin{align*}
 \mu(B_{r/2}(x)\backslash \Omega_\eps) & \le \sum_i \mu(B_{5s_{y_i}}(y_i)) \le C  \sum_i \mu(B_{s_{y_i}}(y_i)) < C \frac{1}{\sqrt{\eps}} \sum_i \int_{B_{s_{y_i}}(y_i)} v \di \mu\\
 &\le  C  \frac{1}{\sqrt{\eps}} \int_{B_r(x)} v \di \mu \le C  \sqrt{\eps} \mu(B_r(x)) \le C \sqrt{\eps} \mu(B_{r/2}(x))
 \end{align*}
where we have used the doubling condition to get the second and the last inequalities, and the fact that $u$ is a reinforced $(k,\eps)$-splitting of $B_r(x)$ to get the fifth one. This shows (A).

Let us prove (C). It follows from the Lebesgue differentiation theorem for doubling metric measure spaces (see e.g. \cite{Heinonen}) that the set of Lebesgue points of $v$ has full measure in $\Omega_\eps$. At any Lebesgue point $y \in \Omega_\eps$ of $v$ we know that
\[
\|G_u(y) - \Id_k\| \le v(y) = \lim\limits_{s \downarrow 0} \fint_{B_s(y)} v \di \mu \le  M_{r/2}v(y) \le \sqrt{\eps},
\]
which yields \eqref{eq:bilipGram}.

We are left with proving (D) namely that for any $y \in \Omega_\eps$,  any $(Y,\dist_Y,\mu_Y,y) \in \mathrm{Tan}_\meas(X,y)$ splits off a factor $\setR^k$.  To this aim, we are going to build a harmonic map $\tilde{u}_\infty : Y \to \setR^k$ such that $G_{\tilde{u}_\infty}(z)=\Id_k$ for $\mu_Y$-a.e.~$z \in Y$. 

For any $s \in (0,r/2)$,  set
\[
\overline{G}_s := \fint_{B_s(y)} G_u \di \mu.
\]
Following a classical argument (see \cite[(4.21)]{Cheeger}, for instance) involving Hölder's inequality, the doubling condition, and the $\uplambda$-Poincaré inequality, 
\begin{align*}
\| \overline{G}_s - \overline{G}_{s/2}\| & \le \fint_{B_{s/2}(y)} \|G_u - \overline{G}_s\| \di \mu\\
& \le A(n) \fint_{B_{s}(y)} \|G_u - \overline{G}_s\| \di \mu\\
& \le A(n) \left( \fint_{B_{s}(y)} \|G_u - \overline{G}_s\|^2 \di \mu \right)^{1/2}\\
& \le A(n)  \uplambda s \left( \fint_{B_{s}(y)} |d G_u|^2 \di \mu \right)^{1/2} \le A(n)  \uplambda s \frac{\eps^{1/4}}{r} \cdot
\end{align*}
This shows that $\{\oG_s\}_{0<s<r/2}$ is a Cauchy sequence, hence it admits a limit $\oG$ as $s \downarrow 0$.  Since
\[
\|\oG - \Id_k\| = \lim\limits_{s \to 0} \|\oG_s - \Id_k\| \le  \lim\limits_{s \to 0} \fint_{B_s(y)}\|G_u - \Id_k\| \le \sqrt{\eps},
\]
we know from Remark \ref{rem:GSsplitting} that there exists a lower triangular $k\times k$ matrix $T$ such that $T\oG {}^t T = \Id_k$ and $\|T\|\le C(n)$ for some generic constant $C(n)$ only depending on $n$. Moreover,  for any $s \in (0,r/2)$,  the previous computation yields
\[
\fint_{B_s(y)} \|G_u - \oG_s\| \di \mu \le A(n)\uplambda s \frac{\eps^{1/4}}{r},
\]
and a telescopic argument gives
\[
\|\oG_s - \oG\| \le C(n) s\frac{\eps^{1/4}}{r},
\]
hence $\tilde{u}:=T\circ u$ satisfies
\begin{equation}\label{eq:control}
\fint_{B_s(y)} \|G_{\tilde{u}} - \Id_k\| \le C(n) s\frac{\eps^{1/4}}{r}\, \cdot
\end{equation}

Now we let $\{s_\alpha\} \subset (0,+\infty)$ be such that $s_\alpha \downarrow 0$ and $\{(X,\dist_\alpha := s_\alpha^{-1} \dist, \mu_\alpha := \mu(B_{s_\alpha}(y))^{-1} \mu, y)\}$ converges to $(Y,\dist_Y,\mu_Y,y)$ in the pointed measured Gromov-Hausdorff topology. Then the maps
\[
u_\alpha :=\frac{1}{s_\alpha} (u-u(y)) : B_{r/2s_\alpha}^{\dist_\alpha}(y) \to \setR^k
\]
are all harmonic and locally $2$-Lipschitz.  By \cite[Proposition E.10]{CMT}, up to extracting a subsequence we may assume that $\{u_\alpha\}$ converges uniformly on compact sets and locally strongly in energy to some harmonic map
\[
\tilde{u}_\infty : Y \to \setR^k.
\]
Then the local strong convergence in energy and \eqref{eq:control} imply that for any $R>0$,
\begin{align*}
\fint_{B_R^{\dist_{Y}}(y)} \|G_{\tilde{u}_\infty}-\Id_k\|  \di \mu_Y & = \lim\limits_\alpha \fint_{B_{R}^{\dist_\alpha}(y)} \|G_{u_\alpha} - \Id_k\| \\
& =  \lim\limits_\alpha \fint_{B_{Rs_\alpha}(y)} \|G_u - \Id_k\| = 0.
\end{align*}
Since $(Y,\dist_Y,\mu_Y)$ is an $\RCD(0,n)$ space, the Functional Splitting Lemma \cite[Lemma 1.21]{AntoBrueSemola}  then yields the conclusion.

\end{proof}

\begin{rem}\label{rk:prop}
The choice of $r/2$ in the previous proof is arbitrary: we can replace it with $\sigma r$ for $\sigma \in (0,1)$ and get the same result.
\end{rem}

\section{Rectifiability of Kato limits}

Let us begin this section with recalling the definitions of bi-Lipschitz map and bi-Lipschitz chart.

\begin{D}
Let $(X,\dist)$ be a metric space, $k$ a positive integer, and $\eps \in (0,1)$.   We say that a map $\phi:X\to\setR^k$ is:
\begin{enumerate}
\item  bi-Lipschitz onto its image if there exists $C\ge 1$ such that $C^{-1} \dist(x,y) \le |\phi(x)-\phi(y)|\le C\dist(x,y)$ for any $x,y \in X$,
\item $(1+\eps)$-bi-Lipschitz onto its image if $(1+\eps)^{-1} \dist(x,y) \le |\phi(x)-\phi(y)|\le (1+\eps) \dist(x,y)$ for any $x,y \in X$.  
\end{enumerate}
Moreover, we call $(1+\eps)$-bi-Lipschitz chart from $X$ to $\setR^k$ any couple $(V,\phi)$ where $V$ is a Borel set of $X$ and $\phi : V \to \setR^k$ is a $(1+\eps)$-bi-Lipschitz map onto its image. 
\end{D}

We now provide a definition of rectifiability for metric measure spaces which is a natural variant of the one introduced in \cite[Definition 5.3]{CheegerColdingIII} and which has notably been used in the setting of $\RCD(K,N)$ spaces \cite{DePhilippisMarcheseRindler,KellMondino,GigliPas}.

\begin{D}
We say that a metric measure space $(X,\dist,\mu)$ is rectifiable if there exists a countable collection $\{(k_i,V_i,\phi_i)\}_i$ where $\{V_i\}$ are Borel subsets covering $X$ up to a $\mu$-negligible set, $\{k_i\}$ are positive integers, and $\phi_i : V_i \to \setR^{k_i}$ is a bi-Lipschitz map such that $(\phi_i)_\#(\mu \measrestr V_i) \ll \leb^{k_i}$ for any $i$. 
\end{D}

According to this definition, our goal in this section is to prove that Kato limit spaces are rectifiable.  Actually, we prove a more precise result which involves the so-called $k$-regular sets.

\begin{D}
For any $k \in \{1,\ldots,n\}$, we define the $k$-regular set of a space $(X,\dist,\mu,o) \in \overline{\mathcal{K}_\meas(n,f,c)}$ as
\[
\cR_k := \{ x \in X : \Tan_\meas(X,x)=\{(\setR^k,\dist_e,\leb^k,0\}\}.
\]
\end{D}

Our main result in this section is the following.

\begin{theorem}\label{th:mainrect}
Let $(X,\dist,\mu,o) \in \overline{\mathcal{K}_\meas(n,f,c)}$. Then the following hold. 
\begin{enumerate}
\item[\normalfont(A)]\label{A} Up to a negligible set, the space $X$ coincides with the union of its $k$-regular sets: 
\begin{equation}\label{eq:essentialdecompo}\mu\left( X \backslash \bigcup_{k = 1}^n \mathcal{R}_k\right) = 0.\end{equation}
\item[\normalfont(B)]\label{B} For any $k \in \{1,\ldots,n\}$ and $\eps \in (0,1)$, there exists a countable family of $(1+\eps)$-bi-Lipschitz charts $\{(V_i^\eps,\phi_i^\eps)\}$ from $X$ to $\setR^k$ such that \[\mu\left(\cR_k \backslash \bigcup_{i} V_i^\eps\right)=0\] and $(\phi_i^\eps)_\#(\mu \measrestr V_i^\eps) \ll \leb^k$ for any $i$.
\end{enumerate}
\end{theorem}

We call \eqref{eq:essentialdecompo} the essential decomposition of $X$. Rectifiability of Kato limit spaces as stated in Theorem \ref{th:main1} is then an obvious corollary of Theorem \ref{th:mainrect}.

The rest of this section is devoted to proving Theorem \ref{th:mainrect}.  Our proof in inspired by \cite{GigliPas,BPS} but contains some simplifications over the arguments presented there.  To keep the notations short,  we write $Y \in \Tan_\meas(X,x)$ instead of $(Y,\dist_Y,\mu_Y,x) \in \Tan_{\meas}(X,x)$.

\subsection{Essential decomposition}

In this subsection, we prove (A) in Theorem \ref{th:mainrect}.  

\begin{proof}[Proof of $\mathrm{(A)}$ in Theorem \ref{th:mainrect}]
First observe that the doubling condition implies the iterated tangent property \cite{LeDonne,GigliMondinoRajala}, meaning that there exists a Borel set $E$ such that $\mu(X\backslash E)=0$ and for any $x\in E$,  any $Y \in \Tan_\meas(X,x)$ and any $y \in Y$, it holds
\begin{equation}\label{eq:iterated}
\Tan_\meas(Y,y) \subset \Tan_\meas(X,x).
\end{equation}
Take $x \in E$ and assume that for some $l \in \{0,\ldots,n\}$ there exists a pointed $\RCD(0,n-l)$ space $Z$ such that $\setR^l \times Z \in  \Tan_\meas(X,x)$.  If $Z$ is not reduced to a singleton,  Gigli's splitting theorem \cite{GigliSplitting} ensures that there exists $z \in Z$ such that any $Z_z \in \Tan_\meas(Z,z)$ splits off an $\setR$ factor, so that \eqref{eq:iterated} implies that there exists a pointed $\RCD(0,n-l-1)$ space $Z'$ such that $\setR^{l+1} \times Z' \in \Tan_\meas(X,x)$.  Then
\[
\setR^{d(x)} \in \Tan_\meas(X,x)
\]
where 
\begin{align*}
d(x) & :=\max\{ 1 \le l \le n \, : \,  \, \text{there exists a pointed $\RCD(0,n)$ space $Z$}\\
& \qquad \qquad \qquad \qquad \quad  \qquad \qquad \qquad\text{such that $\setR^l \times Z \in \Tan_\meas(X,x)$}\}.
\end{align*}
Setting
\begin{align*}
i(x) & :=\min\{ 1 \le l \le n \, : \, \text{there exists a pointed $\RCD(0,n)$ space $Z$}\\
& \qquad \qquad \qquad \qquad \quad  \text{which splits off no $\setR$ such that $\setR^l \times Z \in \Tan_\meas(X,x)$}\},
\end{align*}
we obtain (A) in Theorem \ref{th:mainrect} as a consequence of
\begin{equation}\label{eq:1'}
i(x) = d(x) \qquad \text{for $\mu$-a.e.~$x \in E$.}
\end{equation}
Let us prove \eqref{eq:1'} by contradiction, assuming
\[
\mu(\{x \in E : i(x) < d(x)\}) >0.
\]
Set
\[
\mathfrak{J}_k :=\{x \in E \, : \, d(x)=k\, \,\text{and} \,\,i(x)<k\}
\]
for any $1\le k \le n$,  and note that these sets are measurable as can be proved following the arguments of \cite[Lemma 6.1]{MondinoNaber}. Since
\[
\{x \in E : i(x) < d(x)\} = \bigcup_{1\le k \le n} \mathfrak{J}_k
\]
there exists $k \in \{1,\ldots,n\}$ such that
\[
\mu(\mathfrak{J}_k)>0.
\]
Then $\mathfrak{J}_k$ admits a point with density $1$, that is to say a point $x \in \mathfrak{J}_k$ such that
\begin{equation}\label{eq:1}
\lim\limits_{r \downarrow 0} \frac{\mu(B_r(x) \cap \mathfrak{J}_k)}{\mu(B_r(x))} = 1.
\end{equation}
Since $\setR^{k} \in \Tan_\meas(X,x)$, there exist two infinitesimal sequences $\{\eps_i\}$ and $\{r_i\}$ such that for any $i$ there exists a $(k,\eps_i)$-splitting $u_i$ of $B_{2r_i}(x)$. By propagation of splittings given in Proposition \ref{prop:propagation}, for any $i$ there exists a Borel set $\Omega_i \subset B_{r_i}(x)$ such that 
\begin{equation}\label{eq:density}
 \frac{\mu(B_{r_i}(x)\backslash \Omega_i)}{\mu(B_{r_i}(x))} \le C \sqrt{\eps_i}
\end{equation}
and for any $y \in \Omega_i$ any $Y \in \Tan_\meas(X,y)$ splits off a factor $\setR^{k}$. As a consequence $i(y) \ge k$. This yields $\Omega_i \cap \mathfrak{J}_k = \emptyset$ and \eqref{eq:density} implies
\begin{equation*}
\lim\limits_{i \to \infty}  \frac{\mu(\Omega_i)}{\mu(B_{r_i}(x))} = 1,
\end{equation*}
hence we get a contradiction with \eqref{eq:1}.
\end{proof}

\subsection{Rectifiability of the regular sets: our key result}

In this subsection,  with a view to proving (B) in Theorem \ref{th:mainrect}, we establish the next key technical proposition, where we make use of the almost $k$-regular sets $(\cR_k)_{\delta,r}\subset X$, defined as
\[
(\cR_k)_{\delta,r}:=\left\{x \in X \, :\, \dmGH(B_s(x), \mathbb{B}^k_s)\le \delta s \, \, \, \, \text{for any $s\in (0,r]$}\right\}
\]
for any $\delta,r>0$. Note that each $(\cR_k)_{\delta,r}$ is a closed set. We also define $$(\cR_k)_{\delta}:=\bigcup_{r>0}(\cR_k)_{\delta,r} \subset \left\{x \in X \, : \, 
 \dmGH(B^Y_1(x), \mathbb{B}^k_1)\le \delta\, \,  \, \, \text{for any $Y \in \mathrm{Tan}_\meas(X,x)$} \right\}$$
for any $\delta >0$, and we point out that for any $0<\delta'<\delta$,
\[
(\cR_k)_{\delta} \supset  \left\{x \in X \, : \, 
 \dmGH(B^Y_1(x), \mathbb{B}^k_1)\le \delta'\, \,  \, \, \text{for any $Y \in \mathrm{Tan}_\meas(X,x)$} \right\}.
 \]
Moreover, we have
 $$\cR_k=\bigcap_{\delta}(\cR_k)_{\delta}.$$

 \begin{prop}\label{prop:bilip}
Let $(X,\dist,\mu,o) \in \overline{\mathcal{K}_\meas(n,f,c)}$, $k \in \{1,\ldots,n\}$ and $\eps \in (0,1/2)$ be given. Then there exists $\delta>0$ such that for any $x \in (\cR_k)_{\delta,16r}$ with $r\le \sqrt{T}/16$ and any $s\in (0,r]$ there exist a function $u\colon B_{2s}(x)\rightarrow \R^k$ and  a Borel set $V \subset B_{s}(x)$ such that:
\begin{enumerate}[i)]
\item  $u$ is a  $(k,\eps)$-splitting of $B_{2s}(x)$;
\item  $\mu(B_{s}(x) \backslash V) \le \eps \mu(B_{s}(x))$;
\item  $u$ is an $(\eps\,\sigma)$-GH isometry between $B_{\sigma}(y)$ and $u(y)+\mathbb{B}_{\sigma}^k$ for any $y\in V\cap(\cR_k)_{\delta,16r}$ and any $\sigma\le s/2$;
\item $u$ is $(1+\eps)$-bi-Lipschitz on $V\cap(\cR_k)_{\delta,16r}$;
\item $u_\#\left(\un_{V\cap(\cR_k)_{\delta,16r}}\di\mu\right)\ll \cH^k$.
\end{enumerate}
\end{prop}

In the proof of the last point of this proposition, we use a fundamental result of De Philippis and Rindler \cite[Corollary 1.12]{DePhiRind} which requires the terminology of currents. For the interested reader, we refer to \cite{Federer} or \cite{Simon}.
 
Roughly speaking a current in $\setR^k$ is a differential form whose coefficients are distributions.  To be more precise, let $d$ be a positive integer. A $d$-dimensional current $T$ on $\setR^k$ is a continuous linear map
$$T\colon \cC_0^\infty\left(\setR^k,\Lambda^d(\setR^k)^*\right)\rightarrow \setR.$$
The differential of a $d$-dimensional current $T$ is the $(d-1)$-dimensional current $dT$ defined by
$$dT(\omega):=T(d\omega)$$
for any $\omega \in \cC_0^\infty\left(\setR^k,\Lambda^{d-1}(\setR^k)^*\right)$.  Here we consider only currents with finite mass, that is to say differential forms whose coefficents are finite Radon measures.  Any current with finite mass admits a canonical decomposition
\begin{equation} \label{decomposition}
T(\cdot)= \int_{\setR^k} \langle \cdot, \vec T \rangle \di \|T\|
\end{equation}
where $\|T\|$ is a Radon measure and $\vec T$ is a $\|T\|$-integrable unitary vector field.  In this regard, we shall make use of the following easy lemma, whose proof is omitted for brevity.

\begin{lemma}\label{lem:courants}
Let $\nu$ be a Radon measure on $\setR^d$ and $\vec V$ a square $\nu$-integrable vector field such that $|\vec V (x)|>0$ for $\nu$-a.e.~$x \in \setR^k$. Let $T$ be the one-dimensional current on $\setR^k$ defined by
\[
T(\omega) = \int_{\setR^k} \langle  \vec\omega, \vec V \rangle \di \nu
\]
for any $\omega \in \cC_0^\infty\left(\setR^k,\Lambda^1(\setR^k)^*\right)$. Then $\|T\|$ is absolutely continuous with respect to $\nu$ with density $|\vec V|$ and $\vec T(x) = \vec V(x)/|\vec V(x)|$ for $\nu$-a.e.~$x \in \setR^k$.
\end{lemma}

A current $T$ with finite mass such that $dT$ has finite mass too is called a normal current.  We recall the result by De Philippis and Rindler that we shall use \cite[Corollary 1.12]{DePhiRind}.

\begin{theorem}\label{th:DePhiRin}
Let $\nu$ be a Radon measure on $\setR^k$, and let $\{T_i\}_{1\le i \le k}$ be normal one-dimensional currents on $\setR^k$ such that $\nu \ll \|T_i\|$ for any $i$, and the vectors $\{\vec T_i (x)\}_{1\le i \le k}$ are independent for $\nu$-a.e.~$x \in \setR^k$. Then $\nu \ll \cH^k$.
\end{theorem}

We are now in a position to prove Proposition \ref{prop:bilip}.

\proof  We first prove the first three assertions which are direct consequences of the propagation property of splittings we established in Section 3. Let us set
\[
\tau(n):=\min\left\{ 1, \frac14 C^{-1}(n,1/2), (A(n)\sqrt{C'(n)})^{-1}\right\}
\]
where $A(n)$ is given by the doubling condition \eqref{eq:doubling}, $C(n,1/2)$ is given by Proposition \ref{prop:GHisometry}, and $C'(n)$ is given by Proposition \ref{prop:propagation}.
According to Proposition \ref{prop:GHisometry}, there is some $\delta_1$ such that when
$y\in (\cR_k)_{\delta_1,16r}$, $\sigma\in (0,4r]$ and $v\colon B_{4\sigma}(y)\rightarrow \R^k$ is a $(k,[\tau(n)\eps]^2)$-splitting of $B_{4\sigma}(x)$ then $v$ is an $(\eps\,\sigma)$-GH isometry between $B_{2\sigma}(y)$ and $u(y)+\mathbb{B}_{2\sigma}^k$.

According to Proposition \ref{prop:ExiSplit}, there is a $\delta\le \delta_1$ such that if $x\in (\cR_k)_{\delta,16r}$  and $s\le r$ then there is  $u\colon B_{8s}(x)\rightarrow \R^k$ a reinforced $(k,[\tau(n)\eps]^4)$-splitting of $B_{8s}(x)$.

Now let $x\in  (\cR_k)_{\delta,16r}$ and let $s\in (0,r]$ and $u\colon B_{8s}(x)\rightarrow \R^k$ be a reinforced $(k,[\tau(n)\eps]^4)$-splitting of $B_{8s}(x)$.
With Proposition \ref{prop:propagation}, we find $\Omega\subset B_{4s}(x)$ such that 
$$\mu(B_{4s}(x) \backslash \Omega) \le C'(n)\tau^2(n)\eps^2\,\mu(B_{4s}(x))$$ such that for any
$y\in  \Omega$ and any $\sigma\le s$ then $u$ is a $(k,[\tau(n)\eps]^2)$-splitting of $B_{4\sigma}(y)$. If furthermore $y\in (\cR_k)_{\delta,16r}$ then 
$u$ is an $(\eps\,\sigma)$-GH isometry between $B_{2\sigma}(y)$ and $u(y)+\mathbb{B}_{2\sigma}^k$.

We set $V:=\Omega\cap B_s(x)$. Then
\begin{align*}\mu(B_{s}(x) \backslash V)& \le \mu(B_{4s}(x) \backslash \Omega)\\
&\le C'(n)\tau^2(n)\eps^2\,\mu(B_{4s}(x))\\
&\le A^2(n)C'(n)\tau^2(n)\eps^2\,\mu(B_{s}(x))\\
&\le \eps\mu(B_{s}(x)).\end{align*}

The fourth assertion is a consequence of the third one. Indeed, if $y,z\in V\cap(\cR_k)_{\delta,16r}$, define $2\sigma:=\dist(y,z)\le 2s$. Then, since $u$ is an $(\eps\,\sigma)$-GH isometry between $B_{2\sigma}(y)$ and $u(y)+\mathbb{B}_{2\sigma}^k$, we get
$$\left| |u(y)-u(z)|-\dist(y,z)\right|\le \eps \sigma=\eps \frac{\dist(y,z)}{2}$$
from which follows the desired result.

In order to prove the last point we only need to show that if $K$ is a compact subset of $V\cap(\cR_k)_{\delta,16r} \subset B_s(x)$ with $\mu(K)>0$ then
$$u_\#\left(\un_{K}\di\mu\right)\ll \cH^k.$$

\textbf{Step 1.} To prepare the application of Theorem \ref{th:DePhiRin},  let us introduce a series of Radon measures and discuss some properties of these measures. Set $B:=B_{2s}(x)$. Choose $\{\chi_\ell\} \subset \Lip_c(B,[0,1])$ such that $\chi_\ell \downarrow \un_K$ : for instance for any $\ell$ we may choose $
$$\chi_\ell(\cdot):=\big(1-\ell \dist(K,\cdot)\big)_+$ which has support $K_\ell=\left\{ \dist(K,\cdot)\le \frac1\ell\right\}$. For convenience we also set $\chi_\infty:=\un_K$.
We define the following Radon measures on $\R^k$:
$$\nu_{i,j}^\ell:=u_\#\left(\chi_\ell \Gamma(u_i,u_j)\right)\qquad \text{and}\qquad  \nu^\ell:=u_\#\left(\chi_\ell \mu\right)$$
for $i,j\in \{1,\dots,k\}$ and $\ell\in \N\cup\{\infty\}$ and we also set $$\nu:=u_\#\left(\un_B\mu\right).$$
The coefficients
\[
\frac{\di\Gamma(u_i,u_j)}{\di \mu} = \langle d u_i, d u_j \rangle, \qquad i,j\in \{1,\dots,k\},
\]
of  the Gram matrix map of $u=(u_1,\ldots,u_k)$ are bounded Borel functions, hence there exist bounded Borel functions such that for any $i,j\in \{1,\dots,k\}$ and $\ell\in \N\cup\{\infty\}$,
$$\di\nu_{i,j}^\ell=\rho_{i,j}^\ell \di\nu^\ell.$$
There are also bounded Borel functions $J^\ell$ on $B$ such that $$\di\nu^\ell=J^\ell \di\nu$$
and $J^{\ell+1}\le J^{\ell}\le 1$ for any $\ell \in \setN \cup \{\infty\}$. Moreover,
$$\|J^\ell-J^\infty\|_{L^1(\di\nu)}=\int_{\setR^k} (J^{\ell}- J^\infty)\di\nu=\int_B (\chi_\ell-\chi_\infty)\di\mu\le \mu(K_\ell\setminus K) \to 0$$
so that 
\begin{equation}\label{eq:yieldsalso}
\lim_{\ell\to+\infty} \|J^\ell-J^\infty\|_{L^1(\di\nu)}=0.\end{equation}

\textbf{Step 2.} For any $\ell \in \setN \cup \{\infty\}$,  let $\lambda^\ell$ be the lowest eigenvalue of the symmetric matrix $\left(\rho_{i,j}^\ell\right)$. Our goal is now to establish
  \begin{equation}\label{convergence2}\lim_{\ell\to+\infty}
\|\lambda^\ell- \lambda^\infty\|_{L^1(\di\nu^\infty)}=0\end{equation}
and for $\nu^\infty$-a.e. $p\in \setR^k$,
 \begin{equation}\label{ineq} \lambda^\infty(p)\ge 1-\epsilon. \end{equation}

For any $\xi=(\xi_1,\dots,\xi_k)\in \setR^k$ such that $^t\xi \xi=1$ and any $\ell \in \setN \cup \{\infty\}$, we introduce
 $$\rho_{\xi}^\ell:=\sum_{i,j}\xi_i\xi_j\rho_{i,j}^\ell.$$ Setting $$u_\xi:=\langle \xi,u\rangle$$ we have
 $$ \rho_{\xi}^\ell = \frac{\di u_\#\left(\chi_\ell \Gamma(u_\xi,u_\xi)\right)}{\di \nu^\ell}.$$
In particular, $\{ \rho_{\xi}^\ell(p) \}_{\ell}$ is a non negative non increasing sequence for $\nu^\infty$-a.e.~$p \in \setR^k$. Arguing as we did to get \eqref{eq:yieldsalso} yields
$$\lim_{\ell\to+\infty} \|J^\ell\rho_{\xi}^\ell-J^\infty\rho_{\xi}^\infty \|_{L^1(\di\nu)}=0.$$
Since
\begin{align*}
\|\rho_{\xi}^\ell- \rho_{\xi}^\infty\|_{L^1(\di\nu^\infty)}&=\int_{\setR^k} (\rho_{\xi}^\ell- \rho_{\xi}^\infty) J^\infty\di\nu\\
&=\int_{\setR^k} (J^\ell\rho_{\xi}^\ell- J^\infty\rho_{\xi}^\infty)\di\nu-\int_{\setR^k} (J^\ell- J^\infty)\rho_{\xi}^\ell\di\nu
\end{align*} we also get
$$\lim_{\ell\to+\infty}\|\rho_{\xi}^\ell- \rho_{\xi}^\infty\|_{L^1(\di\nu^\infty)}=0.$$
Using that $\xi\mapsto \rho_{\xi}^\ell$ is quadratic,  by polarization we deduce that for any $i,j$,
$$\lim_{\ell\to+\infty}\|\rho_{i,j}^\ell- \rho_{i,j}^\infty\|_{L^1(\di\nu^\infty)}=0.$$
Up to extraction of a subsequence we can assume that there exists a set $C$ of full $\nu^\infty$ measure such that for any $i,j\in \{1,\dots,k\}$ and $p\in C$, $$ \lim_{\ell\to+\infty}\rho_{i,j}^\ell(p)= \rho_{i,j}^\infty(p).$$
Then for $\nu^\infty$-a.e. $p\in \setR^k$,
\begin{equation}\label{convergence1} \lim_{\ell\to+\infty}\lambda^\ell(p)= \lambda^\infty(p)\end{equation}
  and thus we get \eqref{convergence2}.

For $\nu^\infty$-a.e. $p\in u(K)$, we have 
$$
\rho_\xi^\infty(p)=\lim_{\sigma\to 0} \frac{\int_{K\cap u^{-1}\left(\bB^k_\sigma(p)\right)}  \di\Gamma(u_\xi,u_\xi)}{\mu\left(K\cap u^{-1}\left(\bB^k_\sigma(p)\right)\right)}\, \cdot$$ 
Since $\mu$-a.e.~on $B$ we have
\[
\frac{\di \Gamma(u_\xi,u_\xi)}{\di \mu}={}^t \xi G_u\xi,
\]
from \eqref{eq:bilipGram} in Proposition \ref{prop:propagation} we get $\mu$-a.e.~on $K$:
\[
1-\eps \le \frac{\di \Gamma(u_\xi,u_\xi)}{\di \mu} \le 1+\eps.
\]
Thus for $\nu^\infty$-a.e.~$p\in u(K),$
  \begin{equation}\label{ine}  1-\epsilon \le \rho_\xi^\infty(p)\le 1+\epsilon,\end{equation}
from which follows \eqref{ineq}.

\textbf{Step 3.} Recall that our final goal is to prove that $\nu^\infty \ll \cH^k$. To this aim, we will apply Theorem \ref{th:DePhiRin} for any finite $\ell $ to the  currents
$$T_i^\ell=\sum_{j=1}^k  \nu_{i,j}^\ell dx_j=\sum_{j=1}^k \rho_{i,j}^\ell\nu^\ell dx_j.$$
These are indeed normal currents as for $\psi\in \cC^\infty_0(\setR^k)$,
\begin{align*}
dT_i^\ell(\psi)&=\sum_{j=1}^k\int_B \chi_\ell\, \frac{\partial \psi}{\partial x_j}\circ u\,\di\Gamma(u_i,u_j)\\
&=\int_B \chi_\ell \,\di\Gamma(\psi\circ u,u_i)\text{ using the chain rule}\\
&=-\int_B\psi\circ u \,\di\Gamma(\chi_\ell,u_i) \text{ by the fact that }u_i \text{ is harmonic,}
\end{align*}
hence 
$$dT_i^\ell=-u_\#\left(\Gamma(\chi_\ell,u_i)\right)$$ is a finite Radon measure. Moreover, by Lemma \ref{lem:courants}, the decomposition (\ref{decomposition}) of $T_i^\ell$ is given
by
$$\vec T_i^\ell=\left(\rho_i^\ell \right)^{-1}\left(\rho_{i,1}^\ell,\dots, \rho_{i,k}^\ell\right)$$ with
$$\rho_i^\ell=\left(\sum_{j=1}^k \left(\rho_{i,j}^\ell\right)^2\, \right)^{\frac12} \qquad \text{and} \qquad \|T_i^\ell\|=\rho_i^\ell \nu^\ell.$$
Notice that $\rho_i^\ell\ge \rho_{i,i}^\ell$ hence
$$\rho_{i,i}^\infty  \nu^\infty= \nu_{i,i}^\infty \ll\nu_{i,i}^\ell=\rho_{i,i}^\ell \nu^\ell\ll  \|T_i^\ell\|,$$
and inequality (\ref{ine}) implies that $\nu^\infty$-a.e.~$\rho_{i,i}^\infty \ge 1-\sqrt{\eps}$ so that 
$$\nu^\infty\ll  \|T_i^\ell\|.$$
We remark that for any $\xi=(\xi_1,\dots,\xi_k)\in \setR^k$ unitary it holds
$$\left\langle\, \left(\sum_{i=1}^k \xi_i  \rho_i^\ell\vec T_i^\ell\right),\xi\right\rangle=\rho_\xi^\ell.$$
We set
$$\cB_\ell :=\{p\in \setR^k : \lambda^\ell(p)\le (1-\eps)/2\}.$$
Since $\rho_i^\ell$ are bounded functions, we deduce that 
if  $p\in \setR^k\setminus \cB_\ell$ then  $  \vec T_1^\ell(p),\dots,  \vec T_k^\ell(p)$ is a basis of $\setR^k$.
Applying Theorem \ref{th:DePhiRin} we get
$$\un_{ \setR^k\setminus \cB_\ell}\nu^\infty\ll \cH^k.$$
But the convergence \eqref{convergence2} and the lower bound \eqref{ineq} yield
$$\lim_{\ell\to \infty} \nu^\infty\left(\cB_\ell\right)=0,$$
hence we get $\nu^\infty\ll \cH^k.$
\endproof

\subsection{Rectifiability of the regular sets: end of the proof}

To get (B) in Theorem \ref{th:mainrect} from Proposition \ref{prop:bilip},
we use the following definition, introduced in \cite{BPS}. 
\begin{D}
Let $(X,\dist,\mu)$ be a metric measure space,  $k$ a positive integer and $\eps \in (0,1)$. We call $(\mu,k,\eps)$-rectifiable any Borel set $\Omega \subset X$ for which there exists a countable family of $(1+\eps)$-bi-Lipschitz charts $\{(V_i^\eps,\phi_i^\eps)\}$ from $X$ to $\setR^k$ such that $\mu(\Omega \backslash \bigcup_{i} V_i^\eps)=0$.
\end{D}

According to the previous definition,  we are left with establishing the following.

\begin{prop}\label{prop:rect}
Let $(X,\dist,\mu,o) \in \overline{\mathcal{K}_\meas(n,f,c)}$, $k \in \{1,\ldots,n\}$ and $\eps\in(0,1)$.  Then $\cR_k$ is $(\mu,k,\eps)$-rectifiable.
\end{prop}

To this aim, we prove a lemma which is a consequence of our key Proposition \ref{prop:bilip}.

\begin{lemma}\label{prop:realrect}
Let $(X,\dist,\mu,o) \in \overline{\mathcal{K}_\meas(n,f,c)}$ and $k \in \{1,\ldots,n\}$. Then for any $p \in X$, $R>0$ and $\eps\in (0,1)$, there exists a $(\mu,k,\eps)$-rectifiable set $\Omega_\eps \subset \mathcal{R}_k \cap B_R(p)$ such that $\mu([\mathcal{R}_k \cap B_R(p)] \backslash \Omega_\eps) \le \eps$.
\end{lemma}

\begin{proof}
Let $(X,\dist,\mu,o)\in \overline{\mathcal{K}_\meas(n,f,c)}$, $k \in \{1,\ldots,n\}$, $p \in X$, $R>0$ and $\eps>0$ be given.  Set $\eps':=\eps/\mu(\mathcal{R}_k \cap B_R(p))$.  Let $\delta>0$ be given by Proposition \ref{prop:bilip} applied to $\eps'$. For any $x \in \cR_k$ there exists $r(x)>0$ such that $x \in (\mathcal{R}_k)_{\delta,16r(x)}$.  Apply the Vitali covering lemma for doubling metric measure spaces \cite[Theorem 1.6]{Heinonen} to the set $\mathcal{R}_k \cap B_R(p)$ and the collection of balls $A:=\{B_{r}(x)\}_{x \in \mathcal{R}_k \cap B_R(p), 0<r\le r(x)}$. Then there exists countably many pairwise disjoint balls $\{B_{r_{x_i}}(x_i)\} \subset A$ such that $\mu([\mathcal{R}_k \cap B_R(p)]\backslash \cup_i B_{r_{x_i}}(x_i)) =0$. By Proposition \ref{prop:bilip} for any $i$ there exists a Borel set $V_i \subset B_{r_{x_i}}(x_i)$ which is the domain of a bi-Lipschitz chart and such that $\mu(B_{r_{x_i}}(x_i) \backslash V_i) \le \eps' \mu(B_{r_{x_i}}(x_i) )$.  Set $\Omega_\eps=\cup_i V_i$. Then $\Omega_\eps$ is the union of domains of bi-Lipschitz charts, so it is obviously $(\mu,k,\eps)$-rectifiable. Moreover,
\begin{align*}
\mu([\mathcal{R}_k \cap B_R(p)]\backslash \Omega_\eps) & \le \mu(\cup_i B_{r_{x_i}}(x_i) \backslash V_i)  = \sum_i \mu( B_{r_{x_i}}(x_i) \backslash V_i)\\
&  \le \eps' \sum_i \mu(B_{r_{x_i}}(x_i)) \le \eps' \mu (\mathcal{R}_k \cap B_R(p)) = \eps.
\end{align*}
\end{proof}

We are now in a position to prove Proposition \ref{prop:rect}, and conclude that (B) in Theorem \ref{th:mainrect} is established.

\begin{proof}[Proof of Proposition \ref{prop:rect}]
From the previous lemma, for any $i \in \setN\backslash \{0\}$ there exists a Borel set $\Omega_{\eps,i} \subset \mathcal{R}_k \cap B_i(p)$ which is $(\mu,k,2^{-i}\eps)$-rectifiable and such that $\mu([\mathcal{R}_k \cap B_R(p_i)] \backslash \Omega_{\eps,i}) \le 2^{-i}\eps$.  We set $\Omega_{\eps}:=\bigcup_i \Omega_{\eps, i}$.  Then
\[
\mu(\mathcal{R}_k \backslash \Omega_{\eps}) \le \lim\limits_{i \to +\infty} \mu([\mathcal{R}_k \cap B_i(p)] \backslash \Omega_{\eps}) \le \lim\limits_{i \to +\infty} \mu([\mathcal{R}_k \cap B_i(p)] \backslash \Omega_{\eps,i}) = 0.
\]
Since for any $i$ there exist countably many $(1+\eps)$-bi-Lipschitz charts $\{(V_{i,j}^\eps,\phi_{i,j}^\eps)\}_j$ such that $\mu(\Omega_{\eps,i} \backslash \bigcup_{j} V_{i,j}^\eps)=0$, we get that $\Omega_\eps$ (and then $\mathcal{R}_k$) is $(\mu,k,\eps)$-rectifiable.
\end{proof}

\section{Regularity of non-collapsed strong Kato limits}
\label{sec:Reifenberg}

This section is devoted to the structure and regularity of non-collapsed strong Kato limits. We start by recalling some properties of such spaces,  then show an almost rigidity result that leads to the Reifenberg regularity stated in Theorem \ref{thm:ReifregX}. In the second part of this section, we prove a Transformation Theorem which, together with Theorem \ref{thm:ReifregX} and the results of Section 3, implies Theorem \ref{th:Holderreg}.

\subsection{Non-collapsed strong Kato limits and almost monotone quantity} 
Recall that a manifold $(M^n,g) \in \cK(n,f)$ satisfies a strong Kato bound if the function $f$ is such that 
\begin{equation}
\tag{SK}
\Lambda := \int_0^T \frac{\sqrt{f(s)}}{s} \di s< \infty. 
\end{equation}
Under assumption \eqref{eq:SK}, the volume bound \eqref{eq:volestim} given by Proposition \ref{eq:PI_manifolds} upgrades into the following, as proved in \cite{CMT}. 

\begin{prop}
\label{prop:VB}
Let $(M^n,g) \in \cK(n,f)$ with $f$ satisfying \eqref{eq:SK}. Then there exists $C=C(n, \Lambda)>0$ such that for any $0< s \leq r \leq \sqrt{T}$ we have
$$\frac{\vol_g(B_r(x))}{\vol_g(B_s(x))}\leq C\left( \frac rs\right)^n.$$
\end{prop}

For $v>0$, $(M^n,g,o)$ belongs to $\cK(n,f,v)$ if $f$ satisfies \eqref{eq:SK} and moreover $\vol_g(B_{\sqrt{T}}(o))\geq vT^{\frac{n}{2}}$. Non-collapsed strong Kato limits are elements of the closure $\overline{\cK(n,f,v)}$ with respect to Gromov-Hausdorff topology. As proved in \cite[Theorem 7.1]{CMT}, volume continuity holds for non-collapsed strong Kato limits.

\begin{theorem}
\label{thm:volCont}
Let $\{(M_\alpha, g_\alpha, o_\alpha)\} \subset \cK(n,f,v)$ be a sequence converging in the pointed Gromov-Hausdorff topology to $(X,\dist,o) \in \overline{\cK(n,f,v)}$. Then $(M_\alpha, g_\alpha, \vol_{g_\alpha}, o_\alpha)$ converges to $(X,\dist, \cH^n,o)$ in the pointed measured Gromov-Hausdorff topology. 
\end{theorem}

As a consequence, in this setting the results of Section 3.2 can be revisited. More precisely, if in Theorem \ref{MetaThm}, Propositions \ref{prop:ExiSplit} and \ref{prop:GHisometry}, we replace Kato limits by non-collapsed strong Kato limits, we can assume closeness of balls in the Gromov-Hausdorff topology instead of the measured Gromov-Hausdorff topology. Note that in this case the quantities $\nu$ and $\delta$ also depend on the volume bound $v>0$. 
\medskip

Now let $(X,\dist,o)\in \overline{\cK(n,f,v)}$ and let $H:\R_+\times  X\times X \to \R_+$ be its heat kernel. For any $t>0$ and $x \in X$ we consider 
$$\uptheta(t,x)=(4\pi t)^{\frac n2}H(t,x,x).$$
As we recalled in the introduction, in \cite{CMT} we showed that the map $t\mapsto \uptheta(t,x)$ is almost non-decreasing for all $x \in X$. More precisely, define for any $t \in (0,T]$
$$\upphi(t):=\int_0^t \frac{\sqrt{f(s)}}{s} \di s< \infty.$$
Thanks to the Li-Yau inequality given by Proposition \ref{lem:LY}, we get the following (see also \cite[Corollaries 5.12 and 5.13]{CMT}). 

\begin{prop}\label{monotonetheta} Let $(X,\dist,o)\in \overline{\cK(n,f,v)}$ with $f$ satisfying \eqref{eq:SK}. There is a constant $c_n>0$ depending only on $n$ such that for any $x \in X$ the function
$$t\in (0,T)\mapsto e^{c_n \upphi(t)}\uptheta(t,x)$$ 
is non-decreasing and such that for any $t\in (0,T),$
$$e^{c_n \upphi(t)}\uptheta(t,x)\ge 1.$$
In particular, the limit
$\vartheta(x)=\lim_{t\to 0}\uptheta(t,x)$ is well defined and satisfies $\vartheta(x)\ge 1$. 
\end{prop}

\begin{rem}
In \cite{CMT} we also showed that for all $x \in X$, $\vartheta(x)$ is the inverse of the volume density: $\vartheta(x)^{-1}=\displaystyle\lim_{r \to 0}(\cH^n(B_r(x))/\omega_n r^n),$ where $\omega_n$ is the volume of the Euclidean unit ball. 
\end{rem}

One consequence of \cite{CMT} is that the regular set coincides with the set of points where $\vartheta$ is equal to 1, as we show below.

\begin{prop}\label{prop:cR} Let $(X,\dist,o)\in \overline{\cK(n,f,v)}$ with $f$ satisfying \eqref{eq:SK}. Then
\begin{equation*}
\cR=\{x \in X : \Tan(X,x)=\{(\R^n,\dist_e,0)\}\}=\{x \in X :  \vartheta(x)=1\}.
\end{equation*}
\end{prop}

\begin{proof}
The first equality is a direct consequence of \cite[Theorem 6.2(iii)]{CMT} and of volume continuity as recalled in Theorem \ref{thm:volCont}. As for the second,  \cite[Theorem 7.2]{CMT} ensures that if $(\R^n,\dist_e,0)$ is a tangent cone at $x \in X$, then $\vartheta(x)=1$, so that 
$$\cR \subset \{ x \in X : \vartheta(x)=1\}.$$
To prove the converse inclusion, consider $x \in X$ such that $\vartheta(x)=1$. The proof of \cite[Proposition 6.3]{CMT} ensures that $\vartheta$ is upper semi-continuous. We have then $$1 \leq \liminf_{y \to x} \vartheta(y) \leq \limsup_{y \to x} \vartheta(y) \leq \vartheta(x)=1,$$
so that $\vartheta$ is continuous at $x$. The proof of \cite[Theorem 6.2(iii)]{CMT} then implies that all tangent cones at $x$ are Euclidean, thus $x \in \cR$. 
\end{proof}

For a manifold $(M^n,g)$ satisfying a strong Kato bound, an upper bound on $\uptheta$ at some point  $x$ implies a lower bound on the volume of  $B_{\sqrt{T}}(x)$. 

\begin{lemma}\label{prop:thetanC} Assume that $(M^n,g)$ is a closed manifold in $\cK(n,f)$ with $f$ satisfying \eqref{eq:SK}. There is a constant $v(n)>0$ such that if at some $x\in X$ and $t\le T$ we have
$$\uptheta(t,x)\le 2,$$
then $\nu_g\left(B_{\sqrt{t}}(x)\right)\ge v(n) t^{\frac{n}{2}}.$
\end{lemma}

\proof
Thanks to the heat kernel estimates given by Proposition \ref{Prop:heatKe}, we get
$$ \frac{t^{\frac n2}}{C_n\nu_g(B_{\sqrt{t}}(x))} \le \uptheta(x,t)\le 2,$$
which immediately gives the desired lower bound. \endproof

We are also going to use the following lemma. 

\begin{lemma}
\label{lem:thetaSmallKato}
Let $(M^n,g) \in \cK(n,f)$ for $f$ satisfying \eqref{eq:SK}. For any $\delta \in (0,1)$ there exists $\nu>0$ depending on $\delta, f$ such that if for some $t \in (0, T]$ we have $\kato_t(M^n,g)<\nu,$
then for all $x \in M$ and $s \in (0,t]$ we have $\uptheta(s,x)\leq \uptheta(t,x)(1+\delta)$. 
\end{lemma}

\proof
Assume $\kato_t(M,g)<\nu$ and let $c_n$ be the constant appearing in Proposition \ref{monotonetheta}. Observe that for any $a \in (0,t)$ we can write
$$\int_0^t \frac{\sqrt{\kato_\tau(M^n,g)}}{\tau}\di \tau \leq \int_0^a \frac{\sqrt{f(\tau)}}{\tau}\di \tau + \sqrt{\nu}\log\left(\frac{T}{a}\right).$$
We can choose $a$ depending on $f$ and $\delta$ such that the first addend in the previous inequality is smaller than $\log(1+\delta)/2c_n$. Then we can choose $\nu$ depending on $a$ and $\delta$, thus on $f$ and $\delta$, such that the second addend is also smaller than $\log(1+\delta)/2c_n$. 
By Proposition \ref{monotonetheta}, then we know that for all $x\in M$ and $s \in (0,t]$ 
$$\uptheta(s,x)\leq \uptheta(t,x)\exp\left(c_n\int_s^t \frac{\sqrt{\kato_\tau(M^n,g)}}{\tau}\di \tau\right)\leq \uptheta(t,x)(1+\delta).$$\endproof

\begin{rem} 
\label{rem:thetaSmallKato}
The same argument as in the previous proof implies that for a sequence $\{(M_\ell, g_\ell, o_\ell)\}\subset \cK(n,f,v)$ converging to $(X,\dist, o)\in \overline{\cK(n,f,v)}$ such that $\lim_\ell \kato_t(M_\ell,g_\ell)=0$ for some $t \in (0,T]$, we have that for all $x \in X$ the map $s \mapsto \uptheta(s,x)$ is monotone non-decreasing and satisfies $\uptheta(s,x)\geq 1$ for all $s\in (0,t]$.
\end{rem}

\subsection{Almost rigidity} This subsection is devoted to proving the following almost rigidity for $\uptheta$, which will be the key result to obtain our Reifenberg regularity statement, namely~Theorem \ref{thm:ReifregX}.

\begin{theorem}\label{almostrigidity} For any $\eps>0$ and $A>0$ there exists $\delta>0$ depending only on $f,n,\varepsilon$ and $A$ such that if $(M^n,g)\in\cK(n,f)$, $x\in M$ and  $t\le  T$ satisfy 
$$\kato_t(M,g)\le \delta \quad \text{ and } \quad \uptheta(t,x)\le 1+\delta,$$ 
then 
$$\dGH\left( B_{A\sqrt{t}}(x),\bB^n_{A\sqrt{t}}\right)< \eps A\sqrt{t}.$$
\end{theorem}
In order to prove Theorem \ref{almostrigidity}, we are going to use a contradiction argument, that we sketch here before giving the detailed proof. We will construct a contradicting sequence for which a ball of radius $1$ stays uniformly far from the unit Euclidean ball. Thanks to Lemma \ref{prop:thetanC} such sequence is non-collapsing. Then up to extracting a sub-sequence, we obtain a limit $(X,\dist,x) \in \overline{\cK(n,f,v)}$ such that $B_1(x)$ is at a positive distance from the unit Euclidean ball. We then aim to show that the limit space $(X,\dist)$ is isometric to the Euclidean space. For that, we use the heat kernel rigidity shown in \cite{CT}. More precisely, for a non-collapsed strong Kato limit $(X,\dist, x)\in \overline{\cK(n,f,v)}$ we define
$$\mathbb{P}(t,x,y)=\frac{e^{-\frac{\dist^2(x,y)}{4t}}}{(4\pi t)^{\frac{n}{2}}}.$$
If for all $x,y \in X$ and $t >0$ we have $H(t,x,y)=\mathbb{P}(t,x,y)$, then \cite[Theorem 1.1]{CT} ensures that $(X,\dist)$ is isometric to the Euclidean space. In order to show that $H$ coincides with $\mathbb{P}$, we will rely on the Li-Yau inequality proven in Proposition \ref{lem:LY} and on the fact that, thanks to Remark \ref{rem:thetaSmallKato}, $\uptheta$ is monotone non-decreasing. 

\begin{proof}
We assume by contradiction that the statement is false. Then there exists $\eps, A>0$ such that if we consider the sequence $\delta_\ell=\ell^{-1},$ $\ell \in \N$, we find $t_\ell \leq T$, $(M_\ell, g_\ell)\in \cK(n,f)$ and $x_\ell \in M_\ell$ such that $$\kato_{t_\ell}(M,g)\leq \delta_\ell \mbox{ and } \uptheta(t_\ell,x_\ell) \leq 1+ \delta_\ell,$$
but 
\begin{equation}
\label{eq:pfAR1}
\dGH(B_{A\sqrt{t_\ell}}(x_\ell), \setB^n_{A\sqrt{t_\ell}})> \eps\sqrt{A} t_\ell.
\end{equation}
Observe that if we define $\tilde{ f}(s)= f(sT)$ for all $s\in [0,1]$ and $\tilde g_\ell=t_\ell^{-1}g_\ell$ for any $\ell$, then the rescaling properties of $\kato_t$ and of the heat kernel imply that each $(M_\ell, \tilde g_\ell)$ belongs to $\cK(n,\tilde{ f})$ and
$$\kato_1(M_\ell, \tilde g_\ell)=\kato_{t_\ell}(M_\ell,g_\ell)\leq \delta_\ell, \quad \tilde \uptheta(1,x_\ell)=\uptheta(t_\ell, x_\ell)\leq 1+\delta_\ell.$$
Then up to rescaling we can assume that $t_\ell=1$ for all $\ell \in \N$. 

By Lemma \ref{prop:thetanC}, we also know that there exists $v=v(n)>0$ such that for any $\ell$,
$$\vol_{g_\ell}(B_1(x_\ell))\geq v, $$
so that each $(M_\ell, g_\ell, x_\ell)$ belongs to $\cK(n, f, v)$. Up to extracting a subsequence, $\{(M_\ell, g_\ell, x_\ell)\}$ converges in the pointed Gromov-Hausdorff topology to $(X,\dist,x) \in \overline{\cK(n,f,v)}$. Moreover, convergence of the heat kernel given in Proposition \ref{prop:HKconv} ensures that 
$$\uptheta(1,x)=\lim_\ell \uptheta(1,x_\ell)\leq 1.$$
Thanks to Remark \ref{rem:thetaSmallKato}, we also know that $t \mapsto \uptheta(t,x)$ is monotone non-decreasing and larger than one. We then get for all $s \in (0,1]$, 
\begin{equation}\label{eq:=1}
\uptheta(s,x)=\uptheta(1,x)=1.
\end{equation}

\noindent Because of \eqref{eq:pfAR1}, we also have
\begin{equation}
\label{eq:pfAR2}
\dGH(B_A(x),\setB^n_A)>\eps A.
\end{equation}

Our setting constructed, we aim to prove that the heat kernel of $X$ satisfies 
\begin{equation}
\label{euclHK}
H=\mathbb{P}
\end{equation}
on $\setR_+ \times X \times X$. In order do so, we fix $x \in X$ and we introduce the function 
$$\Phi : \setR_+ \times X \ni (t,y) \mapsto (4\pi t)^{\frac n2} H^2(t/2,x,y).$$

\noindent \underline{\textbf{Step 1.}} We show that $\Phi$ satisfies
\begin{equation}
\label{weak_upper}
4\frac{\partial}{\partial t}\left(\int_X \varphi \Phi \di\cH^n\right)+\int_X \langle d \phi, d \Phi \rangle \di \cH^n\ge 0.
\end{equation}
for any non-negative $\varphi \in \cC_c(X)\cap \cD(\Ch)$.  To this aim, we first observe that 
\begin{align*}
& \phantom{=} \phantom{=}    4\frac{\partial}{\partial t}\left(\int_X \varphi \Phi \di\cH^n\right)\\
& = \int_X \varphi(y)\left( \frac{2n}{t} \Phi(t,y) + 4(4\pi t)^{\frac n2} H(t/2,x,y)\frac{\partial H}{\partial t}(t/2,x,y) \right) \di\cH^n(y).
\end{align*}
Then we use the  definitions of $L$, $H$ and $\Phi$ to get
\begin{align*}
& \phantom{=} \phantom{=}   \int_X \langle d \phi, d \Phi \rangle \di \cH^n = \int_X  \phi  L\Phi  \di \cH^n\\
&=  2(4\pi t)^{\frac n2} \int_X \phi(y) \left( H(t/2,x,y)L_yH(t/2,x,y) - |d_yH(t/2,x,y)|^2  \right) \di \cH^n(y)\\
&=  - 2(4\pi t)^{\frac n2} \int_X \phi(y) \left( H(t/2,x,y)\frac{\partial H}{\partial t}(t/2,x,y) + |d_yH(t/2,x,y)|^2  \right) \di \cH^n(y).
\end{align*}
Adding these two identities yields
$$4\frac{\partial}{\partial t}\left(\int_X \varphi \Phi \di\cH^n\right)+\int_X \langle d \phi, d \Phi \rangle \di \cH^n=
2(4\pi t)^{\frac n2} \int_X\varphi Z \di\cH^n,$$ 
where $Z$ is defined by $$Z(t,y)=\frac{n}{t}H^2(t/2,x,y)+H(t/2,x,y)\frac{\partial H}{\partial t}(t/2,x,y)-|d_yH(t/2,x,y)|^2$$
for any $t\in \setR_+$ and $y \in X$. 
Since $(X, \dist,o)$ is the limit of manifolds $\{(M_\ell,g_\ell)\}$ such that $\kato_1(M_\ell,g_\ell) \to 0$ as $\ell$ goes to infinity, the Li-Yau inequality given by Remark \ref{rem:mieux} holds. Then $Z\geq 0$, this concluding the proof of \eqref{weak_upper}. 
\medskip

\noindent \underline{\textbf{Step 2.}}  We show that for any $t>0$ and $y \in X$,
\begin{equation}
\label{eq:pfAR4}
H(t,x,y)=\mathbb{P}(t,x,y).
\end{equation}
First observe that by definition of $L$,  inequality \eqref{weak_upper} is equivalent to $\Phi$ satisfying
\begin{equation}
\label{upper}
\left( 4\frac{\partial}{\partial t}+L\right) \Phi \geq 0
\end{equation}
in a weak sense. The Gaussian estimate given in Proposition \ref{Prop:heatKe} implies that for any $t>0$,
$$\lim_{\dist(x,y)\to \infty}\Phi(t,y)=0.$$
Moreover,  the fact that $\cH^n(B_{1}(x))\geq v$ and the volume bound given in Proposition \ref{prop:VB} imply that for any $y \in X \backslash \{x\}$,
$$\lim_{t \to 0}\Phi(t,y)=0.$$
By the semi-group law and \eqref{eq:=1}, we know that for any $s\in (0,1]$, 
$$\int_X \Phi(s,y)\di\cH^n(y)=\uptheta(s,x)= 1.$$
As a consequence we get 
$$\lim_{t\to 0}\Phi(t,y)=\delta_x(y).$$
By  inequality \eqref{upper} and the maximum principle, we get that for all  $t\in (0,1]$ and $y \in X$,
$$\Phi(t,y)\geq H(t/4, x,y).$$
But we also have, for all $t \in (0,1]$,
$$1=\int_X \Phi(t,y) \di\mathcal{H}^n(y)=\int_X H(t/4, x,y)\di\mathcal{H}^n(y),$$
 then we obtain, for all $t \in (0,1]$ and $y \in X$,
\begin{equation}
\label{eq:PhiH}
\Phi(t,y)=H(t/4, x,y).
\end{equation}
We now introduce 
$$U(t,x,y)=-4t\log((4\pi t)^{\frac n2}H(t,x,y)).$$
By Varadhan's formula, we know
$$\lim_{\sigma \to 0} U(\sigma, x,y)=-\dist^2(x,y).$$
Because of \eqref{eq:PhiH}, a simple computation shows that for any $s \in (0,1]$ we have
\begin{equation*}
U(s/4,x,y) =
U(s/2, x,y). 
\end{equation*}
As a consequence, for all $s\in (0,1]$,
$$U(s/2,x,y)=\lim_{\sigma \to 0} U(\sigma,x,y)=-\dist^2(x,y).$$
This shows that for all $t \in (0,1/2]$ and $y \in X$
$$H(t,x,y)=\mathbb{P}(t,x,y).$$
Both expressions in this equality are analytic in $t$, hence we get \eqref{eq:pfAR4} for any $t>0$.
\medskip

\noindent \underline{\textbf{Step 3.}} We obtain \eqref{euclHK} and conclude.  Equality \eqref{eq:pfAR4} implies in particular that $\uptheta(t,x)=1$ for all $t>0$ and not only for $t\in (0,1]$. By using the estimate on the derivatives of the heat kernel given in the last point of Proposition \ref{Prop:heatKe}, non-collapsing and the volume bound of Proposition \ref{prop:VB}, we get that there exists a constant $C>0$ such that for any $t>0$ and $z \in X$,
$$|\uptheta(t,x)-\uptheta(t,z)|\leq \frac{C}{\sqrt{t}}\dist(x,z).$$
Then for any $z \in X$,
$$\lim_{t\to +\infty}\uptheta(t,z)=1.$$
Since by Remark \ref{rem:thetaSmallKato} the map $t \mapsto \uptheta(t,z)$ is monotone non-decreasing and larger than one, it must be constantly equal to one. Arguing as in the previous step, the fact that $\uptheta(t,z)=1$ for any $z \in X$ and $t>0$ leads to \eqref{euclHK}.  Then by \cite[Theorem 1.1]{CT}, the strong Kato limit $(X,\dist)$ is isometric to the Euclidean space $(\R^n, \dist_e)$, this contradicting inequality \eqref{eq:pfAR2}. 
\end{proof}

\begin{rem} Theorem \ref{almostrigidity} can be also proven by using \cite[Corollary 1.7]{DPG16}, that is rigidity in Bishop-Gromov inequality for non-collapsed $\RCD(0,n)$ spaces. We chose to provide a self-contained proof independent of $\RCD$ theory. 
\end{rem}

\subsection{Consequences of almost rigidity} As an immediate consequence of Theorem \ref{almostrigidity} and of the convergence of heat kernels given by Proposition \ref{prop:HKconv} we obtain the following. 

\begin{cor}\label{cor:propuptheta} Assume that $f$ satisfies \eqref{eq:SK}. For any $\delta>0$, there is some $\nu>0$ depending only on $f,n$ and $\delta$ such that if $(M^n,g)\in\cK(n,f)$, $x\in M$ and  $t\le  T$ satisfy
$$\kato_t(M,g)\le \nu \quad \text{ and } \quad\uptheta(t,x)\le 1+\nu,$$ 
then for any $y\in B_{\sqrt{t}}(x)$ we have $\uptheta(t,y)\le 1+\delta.$
\end{cor}

By combining Corollary \ref{cor:propuptheta}, the almost monotonicity of $\uptheta$ (Lemma \ref{lem:thetaSmallKato}) with Theorem \ref{almostrigidity}, we get a Reifenberg regularity result for manifolds satisfying a strong Kato bound. 

\begin{cor}\label{thm:ReifregM} Assume that $f$ satisfies \eqref{eq:SK}. For any $\eps>0$, there exists $\nu>0$ depending only on $f,n, \eps$ such that
if $(M^n,g)\in\cK(n,f)$, $x\in M$ and $t\le T$ satisfy
$$\kato_t(M,g)\le \nu \quad \text{ and } \quad\uptheta(t,x)\le 1+\nu$$ then  for any $y\in B_{\sqrt{t}}(x)$ and $s\in (0,\sqrt{t})\colon$
$$\dGH\left( B_{s}(y),\bB^n_{s}\right)\le \eps s.$$ 
\end{cor}

The Reifenberg regularity for non-collapsed strong Kato limits given in Theorem \ref{thm:ReifregX} is then a direct consequence of Corollary \ref{thm:ReifregM}.

We point out a corollary of the almost rigidity statement Theorem \ref{almostrigidity} and of Proposition \ref{prop:ExiSplit} that we use later to obtain Hölder regularity of the regular set of a non-collapsed strong Kato limit.

\begin{cor} 
\label{cor:exiSplGH}
Let $v>0$ and $f$ be a function satisfying \eqref{eq:SK}. For any $\eps>0$ there exists $\delta>0$ depending only $f,n,\eps$ such that if $(M^n,g) \in \cK(n,f,v)$, $x \in M$ and $t \leq T$ satisfy
$$\kato_t(M^n,g)\leq \delta \mbox{ and } \uptheta(t,x)\leq 1+\delta,$$
then there exists an $(n,\eps)$-splitting $u: B_{\sqrt{t}}(x)\to \R^n$.  
\end{cor}

\subsection{Transformation theorem} In order to obtain a quantitative version of Theorem \ref{thm:ReifregX}, we need to prove the following Transformation theorem.

\begin{theorem}[Transformation Theorem]
\label{thm:Transformation}
Let $f$ satisfy \eqref{eq:SK} and $v>0$. There exist a constant $\gamma_n>0$ and $\eps_0 \in (0,1)$ depending on $n,f$ such that for all $\eps \in (0,\eps_0]$ there exists $\delta>0$ depending on $\eps, n, f$ and $v$ such that if $(M^n,g)\in \mathcal{K}(n,f)$, $x\in M$ and $r \in (0,\sqrt{T}]$ satisfy
\begin{enumerate}
\item[\emph{i)}] $\vol_g(B_r(x))\geq vr^n$;
\item[\emph{ii)}] $\kato_{r^2}(M^n,g) \leq \delta;$
\item[\emph{iii)}] for any $s\in (0,r]$, $\dGH(B_{s}(x), \mathbb{B}_{s}^n) \leq \delta s$;
\end{enumerate}
and if $u: B_{r}(x)\to \R^n$ is an $(n,\delta)$-splitting, then for all $s \in (0,r]$ there exists a $n \times n$ lower triangular matrix $T_s$ such that $\|T_s\| \leq (1+\eps)
(r/ s)^{\gamma_n \eps}$ and the map $\tilde u= T_s \circ u$ is an $(n, \eps)$-splitting on $B_s(x)$. 
\end{theorem}

\begin{rem}  Thanks to Lemma \ref{prop:thetanC}, we can reformulate the previous theorem replacing   the non-collapsing assumption i) by $\uptheta(r^2,x)\leq 2$. In this case the choice of $\delta$ will not depend on $v$. \end{rem}

We obtain Theorem \ref{thm:Transformation} as a consequence of the following proposition. 

\begin{prop}
\label{prop:transf1}
Let $(M,g)\in \mathcal{K}(n,f)$. Then there exist $C_n>0$ and \mbox{$\eps_0, \lambda \in (0,1)$} depending only on $n$ such that for all $\eps \in (0,\eps_0]$ there exists $\delta>0$  depending on $n,f, \eps$ such that if for some $r \in (0, \sqrt{T}]$ we have
\[
\mbox{k}_{r^2}(M^n,g) \leq \delta,
\]
a ball $B\subset M$ of radius $r$ satisfies 
\[
\dGH(B, \mathbb{B}_{r}^n) \leq \delta r,
\]
and $u$ is a balanced $(n,\eps)$-splitting of $B$, then there exists a lower triangular matrix $T$ of size $n$ such that $\|T-\Id_n\| \leq C_n\eps$ and the map $\tilde u := T \circ u_{| \lambda B}$ is a balanced $(n,\eps)$-splitting of $\lambda B$. 
\end{prop}

We postpone the proof of Proposition \ref{prop:transf1} and first give a proof of Theorem \ref{thm:Transformation}.  

\begin{proof}[Proof of Theorem \ref{thm:Transformation} given Proposition \ref{prop:transf1}] Let $\eps_0,\lambda$ be as in Proposition \ref{prop:transf1},  and let $\eps \in (0, \eps_0]$. Consider $\eta \in (0,1]$ to be chosen later depending on $n$ and let $\delta=\delta(n, f,\eta \eps)$ be the quantity given by Proposition \ref{prop:transf1}. Assume that 
$$\kato_{r^2}(M^n,g)\leq \delta, \mbox{ for all } s \in (0, r] \, \dGH(B_s(x), \setB^n_s)\leq \delta s.$$
Consider a $(n,\eta\eps)$-splitting $u: B_r(x) \to \R^n$ and $s \in (0,r]$. 

First assume $s \in (\lambda r, r]$. Since $\lambda$ only depends on $n$,  then \eqref{eq:cor_doublement} with $\phi=\|G_u-\Id_n\|$ implies
$$\fint_{B_s(x)}\|G_u-\Id_n\|\di\vol_g < C(n)\eta \eps.$$
If $C(n) \eps_0<1/2$,  Remark \ref{rem:GSsplitting} implies the existence of a lower triangular matrix $T_s$ such that $\|T_s\|\leq 1+C(n)\eta \eps$ and $T_s\circ u : B_s(x) \to \R^n$ is a balanced $(n,(1+C(n)\eta \eps)^2 \eta \eps)$-splitting.  We have no restriction in assuming that $\eps_0$ is lower than $1/4C(n)$,  thus we do it.  Assume also that
\[
\eta \le \frac{16}{25}\, \cdot
\]
Then $T_s \circ u_s$ is a balanced $(n,\eps)$-splitting.

Now assume that there exists some positive integer $\ell$ such that $\lambda^{-\ell}s \in (\lambda r, r]$.  Thanks to assumption iii), we can apply Proposition \ref{prop:transf1} iteratively to get existence of lower triangular matrices $T_0,\ldots, T_l$ such that $\tilde{u} := T_l \circ \ldots \circ T_0 \circ u : B_s(x) \to \setR^n$ is a balanced $(n,\eps)$-splitting and
\[
\|T_j\| \le (1+ C(n) \eta \eps)
\]
for any $j \in \{0,\ldots,l\}$. Set $T=:T_l \circ \ldots \circ T_0$. Then
\[
\|T\| \le (1+ C(n) \eta  \eps)^{l+1}.
\]
Since $\lambda^{-l}s\le r$ implies $l \le \frac{\ln(r/s)}{\ln(1/\lambda)}$,  we get
\[
(1+ C(n) \eta  \eps)^l \le  (r/s)^{\frac{\ln(1+ C(n) \eta  \eps)}{\ln(1/\lambda)}} \le (r/s)^{\frac{C(n) \eps}{\ln(1/\lambda)}}.
\]
Then we set
$$\gamma_n:=\frac{C(n)}{\ln(1/\lambda)}\qquad \text{and} \qquad \eta := \min\left\{\frac{16}{25},\frac{1}{C(n)}\right\}$$
to get $\|T\|\leq (1+\eps)(r/s)^{\gamma_n\eps}$. This concludes the proof.
\end{proof}

\begin{rem}
We point out that, unlike the proof of \cite[Proposition 7.7]{CJN}, which relies on a contradiction argument, we provide a direct proof of the Transformation Theorem. 
\end{rem}

We are left to proving Proposition \ref{prop:transf1}. In order to do so, we need the following property of harmonic maps on $\R^n$. 
\begin{prop}
\label{OscRn}
Let $h : \setB^n \to \setR^k$ be a harmonic function and set 
$$\Lambda:= \fint_{\mathbb{B}^n}\|G_h-\Id_k\|_1\di x.$$
Then there exists a constant $C>0$ depending only on $n$ such that for all $r \in (0,1/2)$
\begin{equation}
\label{eq:harmBr}
\fint_{\mathbb{B}^n_r}\|G_h-\mbox{\small{$\fint_{\mathbb{B}^n_r}G_h$}}\|_1\di x \leq C\Lambda r.
\end{equation}
\end{prop}

\begin{proof} 
For the sake of brevity, we show an analog statement in the case $k=1$: consider a harmonic function $h: \mathbb{B}^n \to \R$ and set 
$$\Lambda_c=\fint_{\mathbb{B}^n} ||dh|^2-c|\di x.$$
Then we show that there exists $C>0$ only depending on $n$ such that for all $r \in (0,1/2)$ we have
\begin{equation}
\label{eq:Osc1}
\fint_{\mathbb{B}^n_r} \left||dh|^2-\fint_{\mathbb{B}^n_r}|dh|^2\right|\di x\leq C\Lambda_c r.
\end{equation}
By arguing as in Lemma \ref{lem:harmGH}, we obtain the following Hessian bound:
\begin{equation}
\label{eq:harm1}
\|\Hess h\|_{L^\infty(\frac 58\setB^n)}\le C_n\sqrt{\Lambda_c}.
\end{equation}
Now we write 
$$h=\ell+\beta,$$
where $\ell$ is the linear part of $h$, namely $\ell(\cdot)=h(0)+dh(0)(\cdot)$, so that $\beta(0)=0$ and $d\beta(0)=0.$  We also have
$$\Hess h=\Hess \beta,$$ 
then from \eqref{eq:harm1} we get, for any $x \in \setB^n_{\frac 58}$,
\begin{equation}
\label{eq:harm2}
|d\beta|(x)\leq C_n\sqrt{\Lambda_c}\, |x|
\end{equation}
Using that the coefficients of $dh$ are harmonic and $d\beta(0)=0$, we obtain
$$\fint_{\setB^n} dh = d\ell \quad \text{ and } \quad  |d\ell|\le \fint_{\setB^n}|dh|.$$
Moreover, for any $r\in (0,1)$ the mean value of  $\langle d\ell,d\beta\rangle$ over $\setB^n_r$ is equal to its value at $0$, thus it is equal to zero. We then get for any $r\in (0,1)$ 
\begin{equation*}
\fint_{r\setB^n} |dh|^2=|d\ell|^2+\fint_{r\setB^n} |d\beta|^2
\end{equation*}
so that
\begin{equation}
\label{eq:harm3}
\fint_{r\setB^n}\left| |dh|^2-\left(\fint_{r\setB^n} |dh|^2\right)\right|\le 2 \fint_{r\setB^n}|d\beta|^2+2 \fint_{r\setB^n}|\langle d\ell,d\beta\rangle|.
\end{equation}
By \eqref{eq:harm2}, the first term in the right-hand side is smaller than $C_n\Lambda_c r^2$. As for the second term, we use
$$2\langle d\ell,d\beta\rangle=|dh|^2-|d\ell|^2-|d\beta|^2$$
to get
\begin{equation*}
2 \fint_{r\setB^n}|\langle d\ell,d\beta\rangle|\le 2 \fint_{r\setB^n}|d\beta|^2+\fint_{r\setB^n}\left| |dh|^2-\left(\fint_{r\setB^n} |dh|^2\right)\right|
\end{equation*}
for any $r\in (0,1)$. Choosing $r=5/8$ gives
$$\fint_{\frac58\setB^n}|\langle d\ell,d\beta\rangle|\le C_n \Lambda_c.$$
Since $\langle d\ell, d\beta \rangle$ is harmonic, elliptic estimates imply the following gradient estimate
$$\| d\langle d\ell, d\beta \rangle \|_{L^{\infty}(\frac 12\setB^n)} \leq C_n\fint_{\frac 58\setB^n}|\langle d\ell,d\beta\rangle|\le C_n \Lambda_c.$$
Then by using that  $\langle d\ell,d\beta\rangle(0)$ vanishes we get for any $x \in \frac 12\setB^n$
$$ |\langle d\ell,d\beta\rangle|(x)\le C_n \Lambda_c |x|.$$
As a consequence, for any $r\in (0, 1/2)$ the second term in \eqref{eq:harm3} is bounded above by $C_n\Lambda r$. We then get the desired inequality
$$\fint_{\setB^n_r}\left| |dh|^2-\left(\fint_{\setB^n_r} |dh|^2\right)\right| \leq C_n\Lambda_c(r^2+r) \leq C_n\Lambda_c r,$$
for any $r\in (0, 1/2)$. 
\end{proof}

We can now prove Proposition \ref{prop:transf1}. 

\begin{proof}[of Proposition \ref{prop:transf1}]
Up to rescaling the distance by a factor $r^{-1}$, we can assume that  $r$ is equal to 1.
 Let $\eps_0, \kappa  \in (0,1)$ and $\lambda \in (0,1/4)$ to be chosen later and which will depend only on the dimension $n$. In what follows we note $C(n)$ for a generic constant which depends only on the dimension $n$ and whose value may change from line to line.
 
 Take $\eps \in (0,\eps_0]$ and let $u$ be a balanced $(n,\eps)$-splitting of a ball $B\subset M$ with radius $1$. We assume that $(M,g)\in \mathcal{K}(n,f)$ and for some $\delta \in (0,1/16n)$, \[
\mbox{k}_{1}(M^n,g) \leq \delta \quad \text{ and } \quad 
\dGH(B, \mathbb{B}_{1}^n) \leq \delta .
\]

 By Proposition \ref{Prop:LipMieux},  we have 
 \begin{equation}\label{Lipu}
 \sup_{\frac{3}{4}B}|du| \leq \left(1+C(n)\delta\right).
 \end{equation}
 If $\delta\le \nu(n,f,v,\kappa\eps,1/2,\lambda)$ 
then by Theorem \ref{MetaThm}, there exists a  harmonic map \mbox{$h: \frac 12 \mathbb{B}^n \to \R^n$} such that $\|dh\|_{L^\infty(\frac 12 \mathbb{B}^n)}\le 2C(n)$ and
\begin{equation}
\label{eq:pf5.12_1.5}
\left|\fint_{\frac12 B} \|G_u - \Id_n\, \|_1 \di\nu_g - \fint_{\frac12\mathbb{B}^n}\|G_h-\Id_n\, \|_1\di x \right| <\kappa \eps ,
\end{equation}
\begin{equation}\label{eq:pf5.12_2}
\left|\fint_{\lambda B} \|G_u - \overline{G_h}\, \|_1 \di\nu_g - \fint_{\lambda\mathbb{B}^n}\|G_h-\overline{G_h}\, \|_1\di x \right| <\kappa \eps ,
\end{equation}
where we have noted $ \overline{G_h}=\fint_{\lambda\mathbb{B}^n} G_h,$ and we introduce similarly  $ \overline{G_u}= \fint_{\lambda B} G_u \di\nu_g$.

We now have that 
\begin{align*}
\fint_{\lambda B} \|G_u - \overline{G_u}\, \| \di\nu_g&\le 2 \fint_{\lambda B} \|G_u - \overline{G_h}\, \| \di\nu_g\\
&\le 2 \fint_{\lambda B} \|G_u - \overline{G_h}\, \|_1 \di\nu_g\\
&\le 2 \fint_{\lambda\mathbb{B}^n}\|G_h-\overline{G_h}\, \|_1\di x +2\kappa \eps,\end{align*} 
 where we have used \eqref{eq:pf5.12_2} and $\|\cdot\|\le \|\cdot\|_1$.
 But using Proposition \ref{OscRn} and then estimate \eqref{eq:pf5.12_1.5}, one gets that 
 \begin{align*}\fint_{\lambda\mathbb{B}^n}\|G_h-\overline{G_h}\, \|_1\di x&\le C(n)\lambda \fint_{\frac12\mathbb{B}^n}\|G_h-\Id_n\, \|_1\di x\\
 & \le C(n)\lambda\left(\kappa\eps+\fint_{\frac12 B} \|G_u - \Id_n\, \|_1 \di\nu_g\right)\\
 & \le C(n)\lambda\left(\kappa\eps+C(n)\eps\right),
 \end{align*} 
where in the last inequality, we have used \eqref{eq:cor_doublement} and $\|\cdot\|_1\le C(n) \|\cdot\|$.
Gathering all the estimates, we get that 
$$\fint_{\lambda B} \|G_u - \overline{G_u}\, \| \di\nu_g\le C(n)\left(\kappa+\lambda\right)\eps.$$
Again \eqref{eq:cor_doublement} implies that 
$$\|\overline{G_u}-\Id_n\|\le \fint_{\lambda B} \|G_u-\Id_n\| \di\nu_g\le C(n,\lambda) \fint_{ B} \|G_u-\Id_n\| \di\nu_g\le C(n,\lambda)\eps.$$
If   $\eps\le \frac{1}{4C(n,\lambda)}$, then by Lemma \ref{lem:GS} there exists a lower triangular matrix $T$ such that
\begin{equation}\label{eq:balanced_0}
T \fint_{\lambda B} G_u \di \vol_g  {}^t T =\Id_n, \quad \| T\| \leq 1+C(n)C(n,\lambda) \eps.
\end{equation}
Then the map $\tilde u = Tu: \lambda B \to \R^n$ satisfies 
\begin{equation*}
\fint_{\lambda B} G_{\tilde u}\di\vol_g =\Id_n,
\end{equation*}
\begin{equation}
\label{eq:pf5.12_0}
\fint_{\lambda B} \| G_{\tilde u} -\Id_n\|\di\vol_g \leq \|T\|^2 \fint_{\lambda B} \left \| G_u - \fint_{\lambda B} G_u \di\vol_g \right \| \di \vol_g\le \|T\|^2\,C(n)\left(\kappa+\lambda\right)\eps,
\end{equation}
and
\begin{equation}
\label{eq:pf5.12_00}
\sup_{\lambda B} |d\tilde u|\le \|T\|\, (1+C(n)\delta).
\end{equation}

We now make the following choices:

  $$\kappa=\lambda=\frac{1}{8C(n)}\text{ and } \eps_0= \frac{1}{4C(n)C(n,\lambda)}$$ and assume that $$\delta=\min\left\{\frac{1}{3C(n)}\,;\, \nu(n,f,v,\eps,\kappa,\lambda)\right\}$$ 
  so that 
  \begin{itemize}
  \item   $\|T\|\le 1+C_n\eps \leq \frac43\le 2 $ by \eqref{eq:balanced_0} and the fact that $\eps\le \eps_0$,
\item $\sup_{\lambda B} |d\tilde u|\le \frac43 (1+C(n)\delta)\le \frac43\frac32=2$ by \eqref{eq:pf5.12_00},
\item $\tilde u$  is a balanced $(n,\eps)$-splitting of $\lambda B$ by \eqref{eq:pf5.12_0}.
\end{itemize}
This concludes the proof.
\end{proof}

\subsection{Hölder regularity} We conclude this section by observing that, under suitable assumptions, the results of the previous sections lead to the following Hölder regularity of almost splitting maps.
\begin{theorem}
\label{thm:strongReifenberg1}  Assume that $f$ satisfies \eqref{eq:SK}. There exists $\eps_0\in (0,1)$ depending only on $f,n$ such that for all $\eps\in (0,\eps_0]$ and for any $\eta \in (0,1)$, there exists $\delta >0$ depending only on $f,n, \eps, \eta$ such that if $(M^n,g) \in \mathcal{K}(n,f)$,  $x \in M$ and $t \in (0,\sqrt{T}]$ satisfy 
$$\kato_t(M^n,g) \leq \delta, \quad \uptheta(t,x) \leq 1+\delta,$$
then any $(n,\delta)$-splitting $u: B_{\sqrt{t}}(x) \to \R^n$, with $u(x)=0$, 
is a diffeomorphism from $B_{(1-\eta)\sqrt{t}}(x)$ onto its image. Moreover, $u$ satisfies for all $y,z \in B_{(1-\eta)\sqrt{t}}(x)$ 
\begin{equation}
\label{eq:strongReif}
(1-\eps)\frac{\dist_g(y,z)^{1+\eps}}{(\sqrt{t})^{\eps}} \leq |u(y)-u(z)| \leq (1+\eps) \dist_g(y,z),
\end{equation}
and we have $\mathbb{B}_{(1-2\eta)\sqrt{t}}^n\subset u(B_{(1-\eta)\sqrt{t}}(x)) \subset \mathbb{B}_{(1-\eta/2)\sqrt{t}}^n.$
\end{theorem} 

As in the proof of \cite[Theorem 7.10]{CJN}, Theorem \ref{thm:strongReifenberg1} follows from the Reifenberg regularity given in Corollary \ref{thm:ReifregM}, Proposition \ref{prop:GHisometry} and the Transformation Theorem \ref{thm:Transformation}. We then refer to \cite{CJN} for the details of the proof. 
\medskip

Theorem \ref{thm:strongReifenberg1} clearly passes to the limit to give an analog statement on non-collapsed strong Kato limits. Now recall that Corollary \ref{cor:exiSplGH} states that if $\uptheta(t,x)$ is close enough to $1$, then there exists an $(n,\eps)$-splitting on a ball around $x$. As a consequence, we obtain: 

\begin{cor}
Assume that $f$ satisfies \eqref{eq:SK}. Let $(X,\dist,o)\in \overline{\cK(n,f,v)}$. For any $\upalpha \in (0,1)$ there exists $\delta$ depending on $\alpha, n$ and $f$ such that for any $x \in X$ satisfying $\vartheta(x) < 1+\delta$ there exist $r \in (0,\sqrt{T})$ and a homeomorphism $u: B_r(x) \to u(B_r(x))\subset \R^n$ such that for all $y,z \in B_r(x)$ we have
$$\upalpha r^{1-\frac{1}{\upalpha}} \dist(y,z)^{\frac{1}{\upalpha}} \leq |u(y)-u(z)| \leq \frac{1}{\upalpha}\dist(y,z)^\alpha r^{1-\upalpha}.$$
\end{cor}

Theorem \ref{thm:ReifregX} is then a consequence of this latter result and of a simple covering argument. 

\appendix
\section*{Appendix}
\renewcommand{\thesubsection}{\Alph{subsection}}

\makeatletter
\renewcommand{\thetheorem}{\thesubsection.\arabic{theorem}}% Update counter printing
\@addtoreset{theorem}{subsection}% Reset theorem counter with every new subsection
\makeatother

\subsection{Codimension 2}

In this section we prove the following.

\begin{theorem} 
\label{thm:codim2}
Assume that \eqref{eq:SK} holds.   Let $(X,\dist,o) \in \strongkatolimits$.  Then the singular set $\cS := X \setminus \cR$ has Hausdorff dimension at most $n-2$.
\end{theorem} 

Consider $(X,\dist,o) \in \strongkatolimits$.  From \cite[Theorem 6.2]{CMT}, we know that the singular set $\cS$ admits a filtration
\[
\cS^0 \subset \ldots \subset \cS^{n-1} = \cS
\]
where
\[
\cS^k:=\{x \in X : \R^\ell\times Z\in \Tan(X,x)\, \Rightarrow \,\ell\le k \}
\]
for any $k \in \{0,\ldots,n-1\}$. Moreover, the Hausdorff dimension of each $\cS^k$ is at most $k$. Thus we are left with proving $\cS^{n-1}=\cS^{n-2}$. 

Let us explain why the latter follows from proving that $\R_+\times \R^{n-1}$ cannot be a tangent cone of $X$ at any $x \in X$.  In \cite[Theorem A]{CMT} we proved that any metric measure tangent cone of $X$ is an $\RCD(0,n)$ metric measure cone. As a consequence,  if $X_x=Z\times \R^{n-1}$  is a tangent cone of $X$ at $x$, since $X$ has Hausdorff dimension at most $n$, then $Z$ is an $\RCD(0,1)$ metric measure cone over some finite set $F$. If $\#F\ge 2$ then $Z$ has at least two ends and as a consequence splits so that necessarily $Z=\R$.  Therefore, we have $\#F=1$ and then $Z=\R_+$, and this is what we aim to prove impossible.

We prove this by contradiction. Assume that there exists $x \in X$ admitting a metric tangent cone isometric to $\setR^+ \times \setR^{n-1}$.  Then there exist pointed closed Riemannian manifolds $\{(M_\alpha,g_\alpha,o_\alpha)\}$ and positive numbers $\{\eps_\alpha\}$ such that $\eps_\alpha\downarrow 0$,
$$ (M_\alpha,\dist_{g_\alpha},o_\alpha) \stackrel{\pGH}{\longrightarrow} (\R_+\times \R^{n-1},\dist_e,0)$$
and 
$$\mbox{k}_t(M_\alpha,g_\alpha)\le f(\eps_\alpha t)$$ for any $\alpha$ and any $t\in (0,1/\eps_\alpha]$.  Set $$\tau \bB^n_+:=\left\{(x_1,\dots,x_n)\in \bB^n_\tau : \ x_1\ge 0\right\}$$ for any $\tau>0$. By arguing as in the proof of \cite[Theorem 7.4]{CMT}, we get harmonic maps
$$\Psi_\alpha=(h_2^\alpha,\dots,h_n^\alpha)\colon B_2(o_\alpha)\rightarrow \R^{n-1}$$ which converge uniformly to $(x_2,\dots,x_n)\colon 2\bB^n_+\rightarrow  \R^{n-1}$ and such that for any $\alpha$,
\begin{enumerate}[i)]
\item $\| d \Psi_\alpha \|_{L^\infty(B_2(o_\alpha))} \le 1+\eps_\alpha$,
\item$\displaystyle \fint_{B_{2}(o_\alpha)}\left\| G_{\Psi_\alpha} - \mbox{Id}_{n-1}\right\| \di \nu_{g_\alpha}\leq \eps_\alpha,$
\item $\displaystyle \fint_{B_{2}(o_\alpha)}| d G_{\Psi_\alpha}|^2 \di \nu_{g_\alpha}\leq \eps_\alpha.$ 
\end{enumerate}
From \cite[Proposition A.1]{CMT}, we get existence of uniformly Lipschitz functions $f_1^\alpha\in \cC^\infty(B_{2}(o_\alpha))$ which converge uniformly to $x_1\colon 2\bB^n_+\rightarrow  \R$.  With no loss of generality, we may assume that
$$\Phi^\alpha : =(f_1^\alpha,h_2^\alpha,\dots,h_n^\alpha)\colon B_2(o_\alpha)\rightarrow 2\bB^n_+$$ is an $\eps_\alpha$-GH isometry.  We are going to modify each $f_1^\alpha$ into a suitable $h_1^\alpha$. To this aim,  we consider a convergent sequence $p_\alpha\in B_{1}(o_\alpha) \to p=(1/2,0,\dots,0)$.  Up to working with $\Phi^\alpha$ modified by an additive constant, we can assume that $$\Phi^\alpha(p_\alpha)=p,$$ and up to considering large enough $\alpha$ only, we can assume that
$$B_{3/8}(p_\alpha)\subset B_1(o_\alpha).$$ For any $\alpha$ let $\tilde f_1^\alpha$ be the harmonic replacement of $f_1^\alpha$ on $B_{3/8}(p_\alpha)$. Then the sequence $\{\tilde f_1^\alpha\}$ is uniformly bounded in energy and in $L^\infty$, and any of its weak sub-limit in energy is equal to $x_1$ on $2\bB^+\setminus B_{3/8}(p)$ and is harmonic on $B_{3/8}(p)$, hence it is equal to $x_1$. Using the energy characterization of harmonic functions and the semicontinuity of the energy, this implies $$\tilde f_1^\alpha\stackrel{E}{\to} x_1.$$
Moreover, the gradient estimate \cite[Lemma 3.6]{CMT} implies that the convergence is uniform on $B_{5/16}(p_\alpha)$. 

For any $\alpha$ let $\chi_\alpha$ be the smooth cut-off function on $M_\alpha$ such that 
$\chi_\alpha=1$ on $B_{9/32}(p_\alpha)$ and  $\chi_\alpha=0$ on $M_\alpha\setminus B_{5/16}(p_\alpha)$ with $\Lip\chi_\alpha\le 64$. Up to extraction of a subsequence, we may assume that $\{\chi_\alpha\}$  converges uniformly to a similar cut-off function on $\R_+\times \R^{n-1}$.  For any $\alpha $ set
$$h_1^\alpha:=\chi_\alpha \tilde f_1^\alpha+(1-\chi_\alpha)f_1^\alpha;$$
then $h_1^\alpha$ is smooth on $B_2(p_\alpha)$ and harmonic on $B_{9/32}(p_\alpha)$. Furthermore, the sequence $\{h_1^\alpha\}$ converges uniformly to $x_1$ on $B_1(p_\alpha)$, and the maps $$h_\alpha:=(h_1^\alpha,h_2^\alpha,\dots,h_n^\alpha)\colon B_1(o_\alpha)\rightarrow \bB_+^n$$ are $\eps_\alpha$-GH isometries which converge uniformly to the identity function.  Moreover,
\begin{enumerate}[i)]
\item $\| d h_\alpha  \|_{L^\infty(B_{17/64}(p_\alpha))} \le 1+\eps_\alpha$,
\item$\displaystyle \fint_{B_{17/64}(p_\alpha)}\left\| G_{h_\alpha} - \mbox{Id}_{n}\right\| \di \nu_{g_\alpha}\leq \eps_\alpha,$
\item $\displaystyle \fint_{B_{17/64}(p_\alpha)}| d G_{h_\alpha}|^2 \di \nu_{g_\alpha}\leq \eps_\alpha.$ 
\end{enumerate}

Let $\{\tau_\alpha\}, \{\rho_\alpha\} \subset (0,1)$ be such that $\tau_\alpha\uparrow 1$,  $\rho_\alpha\uparrow 1/4$, and for any $\alpha$,  $\tau_\alpha^2$ is a regular value of $|h_\alpha|^2$
 and $\rho_\alpha^2$ is a regular value of $|h_1^\alpha-1/2|^2+|\Psi_\alpha|^2$.  For a given $\alpha$, set
$$
\Omega_\alpha:= h_\alpha^{-1}(\mathbb{B}_{\tau_\alpha}^n) \qquad \text{and} \qquad \cU_\alpha:=h_\alpha^{-1}(\mathbb{B}_{\rho_\alpha}^n(p)).
$$ 
Since $h_\alpha(\Omega_\alpha) \subset \bB^n_+$, we know that $h_\alpha\colon \Omega_\alpha\rightarrow \bB^n_{\tau_\alpha}$ is not surjective. Moreover, $h_\alpha( \partial\Omega_\alpha)\subset \partial\bB^n_{\tau_\alpha}$.  Thus for any  regular value  $x\in \tau_\alpha\bB^n_+$ of $h_\alpha$,
 \begin{equation}\label{eq:even}
 \# (h_\alpha^{-1}(\{x\}) \cap  \Omega_{\alpha}) \in 2\N.
 \end{equation}

Let us now consider a sequence $q_\alpha\in \cU_\alpha \to p$ such that each $h_\alpha(q_\alpha)$ is a regular value of $h_\alpha$. 
As each $h_\alpha$ is an $\eps_\alpha$-GH isometry,  for any $q \in \Omega_\alpha$: $$h_\alpha(q)=h_\alpha(q_\alpha)\Longrightarrow \dist_\alpha(q,q_\alpha)\le \eps_\alpha.$$ Hence for large enough $\alpha:$
 $$\left\{q\in \Omega_{\alpha} : h_\alpha(q)=h_\alpha(q_\alpha)\right\}\subset \cU_\alpha.$$
 But the analysis done in the proof of \cite[Theorem 7.4]{CMT} shows that 
 \begin{itemize}
 \item if $\cU_\alpha$ is orientable, then the degree of $h_\alpha\colon \cU_\alpha\rightarrow \bB_{\rho_\alpha}(p)$ is $\pm 1$,
 \item if $\cU_\alpha$ is not orientable and if $\pi_\alpha:\tilde\cU_\alpha\rightarrow \cU_\alpha$ is the 2-fold orientation cover, then the degree of $h_\alpha\circ \pi_\alpha\colon \tilde\cU_\alpha\rightarrow \bB_{\rho_\alpha}(p)$ is $\pm 2$.
 \end{itemize}
 In any case we get
$$\# \left\{q\in \Omega_{\alpha} : h_\alpha(q)=h_\alpha(q_\alpha)\right\} \in 2 \setN +1,$$ which contradicts \eqref{eq:even}.

\subsection{Proof of Theorem \ref{MetaThm}}\label{app:meta} In this section, we obtain Theorem \ref{MetaThm} as a consequence of a contradiction argument and the following result.

\begin{theorem}\label{preMetaThm}
Let $\{(M_\alpha, \dist_{g_\alpha}, \mu_{\alpha}, o_\alpha)\} \subset \mathcal{K}_\meas(n,f,c)$ be converging to $(X,\dist,\mu,o)$ in the pointed measured Gromov-Hausdorff topology.  For some $r \in (0, \sqrt{T}]$, assume that there exists a harmonic function \mbox{$h: B_r(o) \to \R^k$} such that  \mbox{$h(o)=0$} and $\|dh\|_{L^\infty(B_r(o))}\le L$ for some $L >1$.   Let $\eta\in (0,1)$ be given.  Then there exist $C(n,\eta)\ge 1$ and \mbox{$h_\alpha : B_{\eta r}(o_\alpha) \to \R^k$} harmonic with $\|dh_\alpha\|_{L^\infty(B_{\eta r}(o_\alpha))}\le LC(n,\eta)$ and $h_\alpha(o_\alpha)=0$ for any $\alpha$, such that 
\begin{enumerate}
\item for all $s \in (0,\eta r]$
\begin{equation}
\label{eq:average}
\fint_{B_s(o_\alpha)} G_{h_\alpha} \di\mu_{\alpha} \to \fint_{B_s(o)} G_h \di \mu,
\end{equation}
\item for all $s \in (0,\eta r]$ and $A \in \mathcal{M}_k(\R)$
\begin{equation}
\label{eq:matrix}
\fint_{B_s(o_\alpha)}\|G_{h_\alpha}-A\| \di \mu_{\alpha} \to \fint_{B_s(o)}\|G_h-A\|\di \mu.
\end{equation}
\end{enumerate}
\end{theorem}

Before proving it, we need a preliminary lemma.

\begin{lemma}\label{lem:prep_to_meta}
Let $\{(X_\alpha, \dist_\alpha, \mu_\alpha, o_\alpha)\}_{\alpha\in \setN \cup \{\infty\}} \subset \katolimits$ be such that $$(X_\alpha, \dist_\alpha, \mu_\alpha, o_\alpha) \to (X_\infty,\dist_\infty,\mu_\infty,o_\infty)$$ in the pointed measured Gromov-Hausdorff topology.  Consider $r \in (0,\sqrt{T})$.  For any $\alpha$, let $u_\alpha, v_\alpha \in H^{1,2}(B_r(o_\alpha),\dist_\alpha,\mu_\alpha)$ be such that
\begin{enumerate}
\item $u_\alpha \stackrel{L^2(B_r)}{\to} u_\infty$ and $v_\alpha \stackrel{L^2(B_r)}{\to} v_\infty$,
\item\label{assum} $\sup_{\alpha\in \setN}\left(  \int_{B_r(o_\alpha)} \di \Gamma(u_\alpha) \di \mu_{\alpha} \, , \, \int_{B_r(o_\alpha)} \di \Gamma(v_\alpha) \di \mu_{\alpha} \right) < +\infty$.
\end{enumerate}
Then for any $s \in (0,r]$,
\begin{equation}\label{eq:onr}
\fint_{B_s(o_\alpha)} |u_\alpha^2 - v_\alpha^2|\di \mu_\alpha \to \fint_{B_s(o_\infty)} |u_\infty^2 - v_\infty^2|\di  \mu.
\end{equation}
\end{lemma}

\begin{proof} For any $\gamma \in (0,1)$ and $\alpha \in \setN \cup\{\infty\}$, set
\[
u_{\alpha,\gamma}(\cdot) := \fint_{B_\gamma(\cdot)} u_\alpha \di \mu_\alpha \, , \qquad v_{\alpha,\gamma}(\cdot) := \fint_{B_\gamma(\cdot)} v_\alpha \di \mu_\alpha.
\]
Acting as in the proof of \cite[Proposition E.1]{CMT}, it is enough to consider the case $s \in (0,r)$ only. 

We first claim that there exists $C_0>0$ such that for any $\gamma \in (0,r-s)$,
\begin{equation}\label{eq:C1}
\sup_{\alpha \in \setN \cup\{\infty\}} \left| \fint_{B_{s}(o_\alpha)} |u_\alpha^2 - v_\alpha^2| - |u_{\alpha,\gamma}^2 - v_{\alpha,\gamma}^2| \di \mu_\alpha \right|  \le C_0 \gamma.
\end{equation}
Indeed, 
\begin{align*}
\left| \fint_{B_{s}(o_\alpha)} |u_\alpha^2 - v_\alpha^2| - |u_{\alpha,\gamma}^2 - v_{\alpha,\gamma}^2| \di \mu_\alpha \right| & \le \fint_{B_{s}(o_\alpha)} |u_\alpha^2 - u_{\alpha,\gamma}^2| \di \mu_\alpha\\
& + \fint_{B_{s}(o_\alpha)}  |v_\alpha^2 - v_{\alpha,\gamma}^2| \di \mu_\alpha.
\end{align*}
Boundedness in $L^2$ of the averaging operator on doubling spaces (see e.g.~\cite[Theorem 3.5]{Aldaz}) yields the existence of $C_1>0$ such that 
\[
\|u_{\alpha,\gamma}\|_{L^2(B_s(o_\alpha))} \le C_1 \|u_\alpha\|_{L^2(B_r(o_\alpha))}.
\]
Moreover,  the $L^2$ strong convergence of $\{u_\alpha\}$ to $u_\infty$ gives $C_2>0$ such that
\[
 \sup_{\alpha\in \setN \cup \{\infty\}}  \|u_\alpha\|_{L^2(B_r(o_\alpha))} \le C_2.
\]
Finally, the $L^2$ pseudo-Poincaré inequality \cite{CoulhonSC1993} and assumption \eqref{assum} yield the existence of $C_3>0$ such that
\[
\left( \fint_{B_{s}(o_\alpha)} |u_\alpha - u_{\alpha,\gamma}|^2 \di \mu_\alpha\right)^{1/2} \le C_3 \gamma.
\]
Then
\begin{align*}
\fint_{B_s(o_\alpha)} |u_\alpha^2 - u_{\alpha,\gamma}^2| \di \mu_\alpha & \le  \left( \fint_{B_s(o_\alpha)} |u_\alpha - u_{\alpha,\gamma}|^2 \di \mu_\alpha\right)^{1/2}\left( \fint_{B_s(o_\alpha)} |u_\alpha + u_{\alpha,\gamma}|^2 \di \mu_\alpha\right)^{1/2}\\
& \le (1+C_1)C_2 C_3 \gamma.
\end{align*}
This and the symmetry between $u$ and $v$ eventually leads to \eqref{eq:C1}.

We now claim that for any given $\eps>0$ and $\gamma \in (0,1)$, we can choose $\alpha\in \setN$ large enough to ensure
\begin{equation}\label{eq:C2}
 \left|  \fint_{B_s(o_\alpha)}  |u_{\alpha,\gamma}^2 - v_{\alpha,\gamma}^2| \di \mu_\alpha -  \fint_{B_s(o_\infty)}  |u_{\infty,\gamma}^2 - v_{\infty,\gamma}^2| \di \mu_\infty  \right|  \le \frac{\eps}{3}\, \cdot
\end{equation}
The Hölder inequality and a consequence of the doubling condition (see e.g.~\cite[Proposition 1.2, (v)]{CMT}) imply that $\{u_{\alpha,\gamma}\}$ and $\{v_{\alpha,\gamma}\}$ are equicontinuous on balls of radius $B_{(s+r)/2}(o_\alpha)$ for any fixed $\gamma \in (0,1)$. Then $u_{\alpha,\gamma} \to u_\gamma$ and $v_{\alpha,\gamma} \to v_\gamma$ uniformly on $B_s$. This yields \eqref{eq:C2}.

To conclude, take $\eps>0$, choose $\gamma= \eps /(3C_0)$ and then choose $\alpha$ such that \eqref{eq:C2} holds. Then the triangle inequality, \eqref{eq:C1} and \eqref{eq:C2} yield \eqref{eq:onr}.
\end{proof}

\begin{rem}
\label{rem:prep_to_meta}
The previous proof may be easily adapted to show that for any $a \in \R$,
\[
\fint_{B_{r}(o_\alpha)} |u_\alpha^2 - v_\alpha^2-a|\di \mu_\alpha \to \fint_{B_{r}(o)} |u_\infty^2 - v_\infty^2-a|\di  \mu.
\]
\end{rem}

We are now in a position to prove Theorem \ref{preMetaThm} and conclude.

\begin{proof}
We start by treating the case $k=1$.  Consider $\eta'=\eta^{1/2}$ and $\eta''=\eta^{1/3}$ so that $\eta <\eta' < \eta''< 1$.  Then \cite[Proposition E.11]{CMT} ensures the existence of harmonic functions $h_\alpha : B_{\eta'' r}(o_\alpha) \to \R$ uniformly converging to $h|_{B_{\eta''r}(o)}$ on $B_{\eta''r}(o)$ and such that for all $s \in (0, \eta''r]$
\begin{equation}
\label{eq:conv_carré}
\fint_{B_s(o_\alpha)}|dh_\alpha|^2\di\mu_\alpha \to \fint_{B_s(o)}|dh|^2\di\mu.
\end{equation}
By replacing $h_\alpha$ by $h_\alpha - h_\alpha(o_\alpha)$ we can assume that $h_\alpha(o_\alpha)=0$ for all $\alpha$. Moreover, the convergence of $|dh_\alpha|$ given by \eqref{eq:conv_carré} and the fact that  $\|dh\|_{L^\infty(B_r(o))}\le L$ imply that for any large enough $\alpha$
$$\fint_{B_s(o_\alpha)}|dh_\alpha|^2\di\mu_\alpha \leq 2L.$$
We can then apply \cite[Lemma 3.6]{CMT} to get existence of $C(n,\eta)\ge 1$ such that $ \|dh_\alpha\|_{L^\infty(B_{\eta' r}(o_\alpha))}\le L C(n,\eta)$. Now consider $s \in (0,\eta r]$. The previous local Lipschitz bound and the Hessian estimate of \cite[Proposition 3.5]{CMT} yield the uniform Hessian bound 
\begin{equation}
\label{eq:HessBd}
\sup_\alpha \fint_{B_{\eta r}(o_\alpha)} |\nabla dh_\alpha|^2\di \mu_\alpha \leq \frac{C(n,\eta,L)}{r^2} \, \cdot 
\end{equation}
We are then in a position to apply \cite[Proposition E.7]{CMT} and get $L^2(B_{\eta r})$ strong convergence of $\{|dh_\alpha|\}$  to $|dh|$. Then  $\{u_\alpha=|dh_\alpha|\}$  and $\{v_\alpha=0\}$ satisfy the assumptions of Lemma \ref{lem:prep_to_meta}.  We apply it and use Remark \ref{rem:prep_to_meta} to obtain that for all $a\in \R$ and $s\in (0,\eta r]$
$$\fint_{B_s(o_\alpha)}||dh_\alpha|^2-a|\di\mu_\alpha\to \fint_{B_s(o)}||dh|^2-a|\di\mu.$$

We consider now the case $k>1$. Observe that for all $i,j=1, \ldots, k$ we have 
$$(G_{h_\alpha})_{i,j}=\langle (dh_\alpha)_i,(dh_\alpha)_j\rangle = \frac{1}{4}(|d((h_\alpha)_i+(h_\alpha)_j)|^2- |d((h_\alpha)_i-(h_\alpha)_j)|^2).$$
Set
\begin{align*}
f_\alpha =  \frac 12 |d((h_\alpha)_i+(h_\alpha)_j)| , &\quad  \ g_\alpha= \frac 12 |d((h_\alpha)_i-(h_\alpha)_j)|,\\  f =  \frac 12|d(h_i+h_j)|, & \quad g=\frac 12 |d(h_i-h_j)|.
\end{align*}
The sequences $\{f_\alpha\}$ and $\{g_\alpha\}$ satisfy the assumptions of Lemma \ref{lem:prep_to_meta}. This immediately yields \eqref{eq:average}. Moreover, if we consider $A \in \mathcal{M}_k(\R)$ with components $a_{i,j}$, by arguing as above we get for all $i,j=1, \ldots, k$
$$\fint_{B_s(o_\alpha)} |f_\alpha^2-g_\alpha^2-a_{i,j}| \di \mu_\alpha \to \fint_{B_s(o)} |f^2-g^2-a_{i,j}|\di\mu, $$
which is equivalent to 
$$\fint_{B_s(o_\alpha)} | (G_{h_\alpha})_{i,j} - a_{i,j}|\di \mu_\alpha \to \fint_{B_s(o)}|(G_h)_{i,j}-a_{i,j}|\di\mu.$$
This shows \eqref{eq:matrix}. 

\end{proof}

\bibliographystyle{alpha} 
\bibliography{KatoLimits2.bib}

\begin{thebibliography}{DPMR17}

\bibitem[ABS19]{AntoBrueSemola}
Gioacchino Antonelli, Elia Bru{\'e}, and Daniele Semola.
\newblock Volume bounds for the quantitative singular strata of non collapsed
  {R}{C}{D} metric measure spaces.
\newblock {\em Analysis and Geometry in Metric Spaces}, 7(1):158--178, 2019.

\bibitem[Ald19]{Aldaz}
Jes{\'u}s~M Aldaz.
\newblock Boundedness of averaging operators on geometrically doubling metric
  spaces.
\newblock {\em Annales Academiae Scientiarum Fennicae Mathematica},
  44(1):497--503, 2019.

\bibitem[Bam20]{Bamler}
Richard Bamler.
\newblock Structure theory of non-collapsed limits of {R}icci flows.
\newblock ArXiv Preprint: 2009.03243, 2020.

\bibitem[BPS21]{BPS}
Elia Bru{\'e}, Enrico Pasqualetto, and Daniele Semola.
\newblock Rectifiability of {R}{C}{D}({K}, {N}) spaces via splitting maps.
\newblock {\em Annales Fennici Mathematici}, 46(1):465--482, 2021.

\bibitem[BS20]{BrueSemola}
Elia Bru\'{e} and Daniele Semola.
\newblock Constancy of the dimension for {${\rm RCD}(K,N)$} spaces via
  regularity of {L}agrangian flows.
\newblock {\em Comm. Pure Appl. Math.}, 73(6):1141--1204, 2020.

\bibitem[Car19]{C16}
Gilles Carron.
\newblock Geometric inequalities for manifolds with {R}icci curvature in the
  {K}ato class.
\newblock {\em Ann. Inst. Fourier (Grenoble)}, 69(7):3095--3167, 2019.

\bibitem[CC96]{ChCo96}
Jeff Cheeger and Tobias~H Colding.
\newblock Lower bounds on ricci curvature and the almost rigidity of warped
  products.
\newblock {\em Annals of mathematics}, pages 189--237, 1996.

\bibitem[CC97]{ChCo97}
Jeff Cheeger and Tobias~H. Colding.
\newblock On the structure of spaces with {R}icci curvature bounded below. {I}.
\newblock {\em J. Differential Geom.}, 46(3):406--480, 1997.

\bibitem[CC00a]{CheegerColdingII}
Jeff Cheeger and Tobias~H. Colding.
\newblock On the structure of spaces with {R}icci curvature bounded below.
  {II}.
\newblock {\em J. Differential Geom.}, 54(1):13--35, 2000.

\bibitem[CC00b]{CheegerColdingIII}
Jeff Cheeger and Tobias~H. Colding.
\newblock On the structure of spaces with {R}icci curvature bounded below.
  {III}.
\newblock {\em J. Differential Geom.}, 54(1):37--74, 2000.

\bibitem[Che99]{Cheeger}
Jeff Cheeger.
\newblock Differentiability of {L}ipschitz functions on metric measure spaces.
\newblock {\em Geometric \& Functional Analysis GAFA}, 9(3):428--517, 1999.

\bibitem[Che01]{CheegerPisa}
Jeff Cheeger.
\newblock Degeneration of {R}iemannian metrics under {R}icci curvature bounds.
\newblock In {\em Lezione Fermiane}. Accademia Nazionale dei Lincei, Scoula
  Normale Superiore, Pisa, 2001.

\bibitem[CJN21]{CJN}
Jeff Cheeger, Wenshuai Jiang, and Aaron Naber.
\newblock Rectifiability of singular sets of noncollapsed limit spaces with
  {R}icci curvature bounded below.
\newblock {\em Ann. of Math. (2)}, 193(2):407--538, 2021.

\bibitem[CMT21]{CMT}
Gilles Carron, Ilaria Mondello, and David Tewodrose.
\newblock Limits of manifolds with a {K}ato bound on the {R}icci curvature.
\newblock ArXiv Preprint: 2102.05940, 2021.

\bibitem[CN12]{ColdingNaber}
Tobias~Holck Colding and Aaron Naber.
\newblock Sharp h{\"o}lder continuity of tangent cones for spaces with a lower
  ricci curvature bound and applications.
\newblock {\em Annals of mathematics}, pages 1173--1229, 2012.

\bibitem[CN15]{CheegerNaber}
Jeff Cheeger and Aaron Naber.
\newblock Regularity of {E}instein manifolds and the codimension 4 conjecture.
\newblock {\em Ann. of Math. (2)}, 182(3):1093--1165, 2015.

\bibitem[Col97]{C97}
Tobias~H Colding.
\newblock Ricci curvature and volume convergence.
\newblock {\em Annals of mathematics}, 145(3):477--501, 1997.

\bibitem[CSC93]{CoulhonSC1993}
Thierry Coulhon and Laurent Saloff-Coste.
\newblock Isop\'{e}rim\'{e}trie pour les groupes et les vari\'{e}t\'{e}s.
\newblock {\em Rev. Mat. Iberoamericana}, 9(2):293--314, 1993.

\bibitem[CT22]{CT}
Gilles Carron and David Tewodrose.
\newblock A rigidity result for metric measure spaces with {E}uclidean heat
  kernel.
\newblock {\em Journal de l'{\'E}cole Polytechnique --- Math{\'e}matiques},
  Tome 9:101--154, 2022.

\bibitem[Don11]{LeDonne}
Enrico~Le Donne.
\newblock Metric spaces with unique tangents.
\newblock In {\em Annales Academiae Scientiarum Fennicae. Mathematica},
  volume~36, 2011.

\bibitem[DPG16]{DPG16}
Guido De~Philippis and Nicola Gigli.
\newblock From volume cone to metric cone in the nonsmooth setting.
\newblock {\em Geom. Funct. Anal.}, 26(6):1526--1587, 2016.

\bibitem[DPMR17]{DePhilippisMarcheseRindler}
Guido De~Philippis, Andrea Marchese, and Filip Rindler.
\newblock On a conjecture of {C}heeger.
\newblock In {\em Measure theory in non-smooth spaces}, Partial Differ. Equ.
  Meas. Theory, pages 145--155. De Gruyter Open, Warsaw, 2017.

\bibitem[DPR16]{DePhiRind}
Guido De~Philippis and Filip Rindler.
\newblock On the structure of {A}-free measures and applications.
\newblock {\em Annals of Mathematics}, pages 1017--1039, 2016.

\bibitem[Fed14]{Federer}
Herbert Federer.
\newblock {\em Geometric measure theory}.
\newblock Springer, 2014.

\bibitem[Gig13]{GigliSplitting}
Nicola Gigli.
\newblock The splitting theorem in non-smooth context.
\newblock {\em arXiv:1302.5555}, 2013.

\bibitem[Gig15]{GigliMAMS}
Nicola Gigli.
\newblock {\em On the differential structure of metric measure spaces and
  applications}.
\newblock American Mathematical Soc., 2015.

\bibitem[Gig18a]{GigliLectureNotes}
Nicola Gigli.
\newblock {L}ecture {N}otes on {D}ifferential {C}alculus on {R}{C}{D} {S}paces.
\newblock {\em Publications of the Research Institute for Mathematical
  Sciences}, 54(4):855--918, 2018.

\bibitem[Gig18b]{Gigli}
Nicola Gigli.
\newblock {\em Nonsmooth differential geometry--an approach tailored for spaces
  with Ricci curvature bounded from below}, volume 251.
\newblock American Mathematical Society, 2018.

\bibitem[GMR15]{GigliMondinoRajala}
Nicola Gigli, Andrea Mondino, and Tapio Rajala.
\newblock Euclidean spaces as weak tangents of infinitesimally {H}ilbertian
  metric measure spaces with {R}icci curvature bounded below.
\newblock {\em Journal f{\"u}r die reine und angewandte Mathematik (Crelles
  Journal)}, 2015(705):233--244, 2015.

\bibitem[GP21]{GigliPas}
Nicola Gigli and Enrico Pasqualetto.
\newblock Behaviour of the reference measure on {R}{C}{D} spaces under charts.
\newblock {\em Communications in Analysis and Geometry}, 29(6):1391--1414,
  2021.

\bibitem[Gri95]{GriderHeat}
Alexander Grigoryan.
\newblock Upper bounds of derivatives of the heat kernel on an arbitrary
  complete manifold.
\newblock {\em Journal of Functional Analysis}, 127(2):363--389, 1995.

\bibitem[Hei12]{Heinonen}
Juha Heinonen.
\newblock {\em Lectures on analysis on metric spaces}.
\newblock Springer Science \& Business Media, 2012.

\bibitem[HKST15]{HKST}
Juha Heinonen, Pekka Koskela, Nageswari Shanmugalingam, and Jeremy~T Tyson.
\newblock {\em Sobolev spaces on metric measure spaces}.
\newblock Number~27. Cambridge University Press, 2015.

\bibitem[KM18]{KellMondino}
Martin Kell and Andrea Mondino.
\newblock On the volume measure of non-smooth spaces with {R}icci curvature
  bounded below.
\newblock {\em Ann. Sc. Norm. Super. Pisa Cl. Sci. (5)}, 18(2):593--610, 2018.

\bibitem[MN19]{MondinoNaber}
Andrea Mondino and Aaron Naber.
\newblock Structure theory of metric measure spaces with lower {R}icci
  curvature bounds.
\newblock {\em Journal of the European Mathematical Society}, 21(6):1809--1854,
  2019.

\bibitem[RS17]{RoseStollmann}
Christian Rose and Peter Stollmann.
\newblock The {K}ato class on compact manifolds with integral bounds on the
  negative part of {R}icci curvature.
\newblock {\em Proc. Amer. Math. Soc.}, 145(5):2199--2210, 2017.

\bibitem[Sim14]{Simon}
Leon Simon.
\newblock Introduction to geometric measure theory.
\newblock {\em Tsinghua Lectures}, 2014.

\bibitem[Vil09]{Villani}
C\'{e}dric Villani.
\newblock {\em Optimal transport: Old and New}, volume 338 of {\em Grundlehren
  der Mathematischen Wissenschaften [Fundamental Principles of Mathematical
  Sciences]}.
\newblock Springer-Verlag, Berlin, 2009.
\newblock Old and new.

\end{thebibliography}

\end{document}